\documentclass[10pt,openany]{bookk}
\usepackage{amsmath,amsthm,amsfonts,amssymb,amscd}
\usepackage[verbose,includemp=false, paperwidth=15.50cm, paperheight=21cm, body={10.50cm,16cm}, rmargin= 1.5cm, twosideshift=0pt]{geometry}
\usepackage[cam,letter,center]{crop}
\usepackage[latin1]{inputenc}
\usepackage{cite}

\newcommand{\ii}{\sqrt{-1}}
\newcommand{\Rmin}{R_{\text{min}}}
\newcommand{\Rmax}{R_{\text{max}}}
\newcommand{\gr}{\nabla}
\newcommand{\RF}{Ricci flow \;}
\newcommand{\NRF}{normalized Ricci flow \;}
\newcommand{\KRF}{K\"ahler Ricci flow \;}

\newcommand{\Ric}{\text{Ric}}
\newcommand{\Vol}{\text{Vol}}
\newcommand{\Ker}{\text{Ker}}
\newcommand{\ibar}{\bar {i}}
\newcommand{\kbar}{\bar {k}}
\newcommand{\jbar}{\bar {j}}
\newcommand{\vbar}{\bar {v}}
\newcommand{\zbar}{\bar {z}}
\newcommand{\lbar}{\bar {l}}
\newcommand{\pbar}{\bar {p}}

\newcommand{\kahler}{K\"ahler     }
\newcommand{\KE}{K\"ahler-Einstein     }
\newcommand{\ddzi}{\frac{\partial}{\partial z^i}}

\providecommand{\ip}[1]{\ensuremath{\langle #1\rangle}}

\newtheorem{thm}{Theorem}[section]
\newtheorem{defin}{Definition}[section]
\newtheorem{prop}{Proposition}[section]
\newtheorem{lemma}{Lemma}[section]
\newtheorem{cor}{Corollary}[section]

\newcommand{\bede}{\begin{defin}}
\newcommand{\eede}{\end{defin}}

\newcommand{\bethm}{\begin{thm}}
\newcommand{\eethm}{\end{thm}}

\newcommand{\beit}{\begin{itemize}}
\newcommand{\eeit}{\end{itemize}}

\newcommand{\bedes}{\begin{description}}
\newcommand{\eedes}{\end{description}}

\newcommand{\beee}{\begin{enumerate}}
\newcommand{\eeee}{\end{enumerate}}

\newcommand{\beeq}{\begin{equation}}
\newcommand{\eeeq}{\end{equation}}

\def\ddt {\frac{\partial}{\partial t}}

\def\PP {{\mathbb P}}

\def\RR {{\mathbb R}}
\def\EE{{\mathbb E}}
\def\HH{{\mathbb H}}
\def\CC {{\mathbb C}}

\def\La{\Lambda}
\def\De{\Delta}
\def\Om{\Omega}
\def\Ga{\Gamma}
\def\si{\sigma}
\def\la{\lambda}
\def\de{\delta}
\def\om{\omega}
\def\ga{\gamma}
\def\ve{\varepsilon}
\def\vp{\varphi}

\title{Introduction  to Evolution Equations in Geometry}

\author{Bianca Santoro}
\date{May 15th, 2009}

\begin{document}

\maketitle

\clearpage\mbox{}\clearpage

\tableofcontents
\setcounter{tocdepth}{2}

\chapter*{Preface}

These are the very unpretentious lecture notes for the minicourse
{\em Introduction to evolution equations in Geometry},
a part of the {\em  Brazilian Colloquium of Mathematics},
to be held at IMPA, in July of $2009$.

I have aimed at providing a first introduction to the main general ideas
on the study of the \RF,
as well as guiding the reader through the steps of \kahler
geometry for the understanding of the complex version of the \RF.
Most of the results concerning the (real) \RF  are due to Hamilton,
and the \KRF results are mainly due to Cao, but while researching
during the writing of my text, I found that the expositions
in \cite{Ben1} and \cite{Topping} are really clarifying.

There are extremely important  and deep aspects to this theory, concerning
   \RF with surgery and the work by Perelman, which are not discussed
   in this book.

I plan to keep improving these notes, and the updates will be available
at {\it www.math.duke.edu/$\sim$bsantoro}.
Meanwhile, I invite the reader to send suggestions, comments and possible corrections
to

\bigskip

\hspace{2.5cm}
{\bf bsantoro@math.duke.edu}

\bigskip

A last remark is that the reader will notice that on this text, constants
will be represented with the same symbol $C$, with rare exceptions.

\clearpage\mbox{}\clearpage

\chapter{Introduction}
\label{Introduction}

A classical problem in geometry is to seek special metrics
on a manifold.
From our differential geometry classes, we recall that the topology
of the manifold plays a central role in determining the presence of
metrics with special curvature. For the simple case of surfaces,
the {\bf Gauss-Bonnet Theorem} relates the total curvature of a
surface $M$
with its Euler characteristic:
$$
\int_M K dS = 2 \pi \chi(M).
$$
This expression tells us that every time we try to
"flatten" a region of our surface, we must pay the price of
modifying the curvature somewhere else as well. Another direct consequence
is that a sphere will never admit a flat metric.

In dimension $2$, the {\bf Uniformization Theorem} states that
every surface admits a metric of constant curvature $-1, 0$ or $1$,
according to its genus.

It is hence a  natural question to ask if such a classification
extends to higher dimensions.

Already in dimension $3$, a naive attempt of an immediate generalization of the
Uniformization Theorem fails, as $S^2 \times S^1$ cannot admit a
metric of constant curvature. So, a classification would only be possible if
we broaden the conditions, as well as narrow the types of manifolds in question.

As we shall see in Chapter~\ref{GeomConj}, Thurston's Decomposition Theorem
tells us that, for any $3$-manifold $M$, there is a way to cut it along special tori
and spheres (and gluing $3$-balls in the remaining boundary pieces) in such a way that
the remaining pieces admit "canonical geometries", in a way to be made precise in later chapters.

And here the Ricci flow enters the stage. Hamilton suggested
that such a decompotion in pieces with nice geometric structures
could be obtaining by flowing any initial metric on a manifold
through a smartly chosen equation. Understanding the
limiting metrics and the structure of eventual singularities that
may occur would provide a complete  understanding of the problem.

Now, what is a good choice for a flow?
The strategy to produce a nice evolution equation is to
minimize the right choice of energy functional.
Keeping in mind that any expression in the first derivatives
of $g$ would vanish in normal coordinates,
a first choice for such an energy functional would be
$$
E = \int_M R d\mu,
$$
where $R$ denotes the scalar curvature of a metric $g$.

The associated flow equation for minimizing $E$
is
$$
\ddt g(t) = \frac{2}{n} R g(t) - 2 \Ric(t),
$$
where $\Ric$ is the Ricci curvature of the metric $g(t)$.

Unfortunately, such an equation would imply that the evolution
equation for the scalar curvature is given by
$$
\ddt R  = -\left(\frac{n}{2} -1 \right)\De R +
|\Ric|^2 - \frac{1}{2}R^2,
$$
which is a backwards heat equation on $R$, which may have  no solutions even for short time.
Therefore, we need to be smarter about the choice of evolution equation.

Consider
\beeq
\label{e.NRF}
\begin{cases}
\ddt g(t) = \frac{2}{n} r g(t) - 2 \Ric(t)\\
g(0) = g_0,
\end{cases}
\eeeq
where $r$ denotes the average of the scalar curvature on the manifold $M$.

In fact, the first factor in the right-hand side of (\ref{e.NRF}) is just a normalization
that will guarantee that the volume of the manifold remains constant along the flow.
We claim that solutions to (\ref{e.NRF}) correspond (via rescalings) to solutions to the
unnormalized \RF equation
\beeq
\label{e.RF}
\begin{cases}
\ddt g(t) = - 2 \Ric(t)\\
g(0) = g_0.
\end{cases}
\eeeq

To see this, let $g$ be a solution to (\ref{e.RF}), and choose a function
$\psi$ such that the metric $\tilde g = \psi g$ has volume
$\int d\tilde \mu = 1$. With this normalization, you can check that
$$
{\tilde \Ric}  = \Ric,  {\hspace{1cm}} {\tilde R} = \psi^{-1} R, {\hspace{1cm}} {\tilde r} = \psi^{-1} r,
$$
where the geometric objects denoted with a $\sim$ on top are the ones related to $\tilde g$.

Let $\tilde t  = \int \psi(t) dt$ be  a rescaling of time. Then,
$$
\frac{\partial}{\partial {\tilde t}}\tilde g  = \ddt g + (\frac{d}{dt} \log \psi) g.
$$

Also, according to our definition,
$$
\int d\mu = \psi^{-n/2}  {\hspace{1cm}} \text{and} \hspace{1cm} \ddt \log \mu = \frac{1}{2}g^{ij}\ddt g_{ij} = -R
$$
and since
$$
\frac{d}{dt} \log(\int d\mu) = -r,
$$
this implies that
$$
\frac{\partial}{\partial {\tilde t}}\tilde g  = \frac{2}{n} {\tilde r} {\tilde g}  - 2 {\tilde \Ric},
$$
{\em i.e.}, $\tilde g$ is a solution to the \NRF (\ref{e.NRF}).

\bigskip

In dimension $2$, the \RF has a solution that exists for all times,
and it converges to the special metric of constant curvature of the surface,
providing a different proof for the Uniformization Theorem.

More generally, the \RF may develop singularities in finite time.
Nevertheless, the behavior serves  for understanding better the underlying
topology.

The current strategy is: we stop the flow once the singularity occurs,
perform a very careful surgery on the manifold and we restart the flow
on the "desingularized" manifold. Thanks to the groundbreaking work of
Perelman \cite{Perelman1}, \cite{Perelman2}, it is known that
this surgery procedure will only be performed a finite number of times,
and will provide the Thurston's decompotion.
Due to the introductory flavor of this set of notes, we will not
devote time to explain the beautiful theory of singularities on \RF.
We refer the reader to \cite{Ben2} for a very good exposition of this material.

\bigskip

In order to develop some very basic intuition to the \RF, let us study the simple example of a round sphere, with a metric $g_0$ such that
$$
\Ric (g_0) = \la g_0
$$
for some constant $\la> 0$.  We invite the reader to check that
 the solution of the \RF equation is
$$
g(t) = (1 - 2\la t)g_0.
$$
So, we see that the unnormalized \RF
shrinks a round sphere to a point in time $T  = (2\la)^{-1}$.

Note also that if $g_0$ were a hyperbolic metric (of constant sectional
curvature $-1$), then
$$
\Ric (g_0) = -(n-1) g_0,
$$
and so
$$
g(t) = (1 + 2(n - 1) t)g_0,
$$
which shows that the manifold expands homothetically for all times.

For a beautiful exposition on the intuition for the \RF, we refer
the reader to \cite{Topping}.

\chapter{The Geometrization Conjecture}
\label{GeomConj}

Our goal for this chapter is to describe the concept of a
model geometry,  understand Thurston's classification of
homogeneous geometries, and state the Geometrization Conjecture in detail.

\section{Introduction}

 In order to provide a better picture, we shall describe the concept of a
 model geometry for surfaces.

 Let $X$ be either $S^2$ (the sphere), $\EE^2 $ (Euclidean plane) or $\HH^2$
 (hyperbolic plane).

 \bede
 Let $F$ be a closed surface. If $F$ can be described as
 the quotient $X / \Ga$, where $\Ga$ is a subgroup of isomorphisms
 of $X$, such that the map $X \rightarrow F$ is a covering space,
 then we say that $F$ has a {\bf geometric structure modeled on $X$}.

 \eede

The {\bf Uniformization Theorem} tells us that every closed surface admits
a geometric structure, and due to Gauss-Bonnet Theorem, this structure must
be unique, and is determined by the Euler characteristic of the surface.

A natural question that arises is whether this classification can be extended to
higher dimensions. The simple example of  $S^2 \times S^1$ tells us that a naive extension
of such classification fails, for it cannot be covered by any of $\EE^3$, $S^3$ or $\HH^3$,
as the stabilizer of a point is not in $O(3)$.

First, we need to explain what do we mean by a geometric structure.

\bede
A metric on a manifold $M$ is {\bf locally homogeneous} if
for any $x, y \in M$, there exist neighborhoods
$U $ of $x$ and $V $ of $y$, and an isometry $U \rightarrow V$
mapping $x$ to $y$.
\eede

Intuitively, a locally homogeneous manifold looks the same in a
neighborhood of any point.

\bede
We say that $M$ {\bf admits a geometric structure} if $M$ admits a
complete, locally homogeneous metric.

\eede

It is a theorem by Singer \cite{Singer} that if a manifold happens
to be simply connected, then being locally homogeneous is equivalent
to being homogeneous. Therefore, up to passing to the universal cover,
we may restrict ourselves to the study of homogeneous geometries.

However, many manifolds do not admit a geometric structure. For example,
any non-trivial connected sum does not admit a geometric structure, with the
exception of $\PP^3 \# \PP^3$. So we should be a little more subtle in our
attempt to classify $3$-manifolds.

The famous theorem by Thurston roughly says that any compact,
orientable $3$-manifold $M$ can be cut along disjoint embedded
$2$-spheres and tori in such a way that after gluing $3$-balls
to all boundaries, each piece  admits a geometric structure. We will
describe this statement in further detail on the next section.

For a detailed description of the eight geometries, we refer the
reader to \cite{PeterScott}.

\section{Thurston's Classification of $3$-dimensional geometries}

In this section, we would like to present an outline of the proof that
there are only $8$ distinct geometric structures, and that once a
manifold admits one of them, it is unique.

Let $M$ be a closed $3$-manifold, and let $X$ be its universal cover.
As we have seen, $M$ admits a geometric structure if $X$ has a complete
homogeneous metric. For this case, the isometry group of $X$ acts transitively,
with compact point stabilizers.

\bede
Throughout this chapter, a
 {\bf geometry} shall mean a pair $(X, G)$, where  $X$ is a simply connected manifold,
 and $G$ acts on $X$ with compact point stabilizers.
\eede

We will only consider maximal geometries (when $G$ is maximal), and we need to impose the
condition that there must exist a subgroup $H$ of $G$ which acts on $X$ as a covering group
and has compact quotient.

\bethm
(Thurston)
\label{t.Thurston}
Any maximal, simply connected $3$-dimensional geometry which admits a compact quotient
is equivalent to a pair $(X, G)$, $G = $Isom$(X)$, where $X$ is one of the following:
$$
\EE^3, S^3, \HH^3, S^2 \times \RR, \HH^2 \times \RR, \tilde{SL(2, \RR)}, Nil, Sol.
$$

\eethm

The remainder of this section will be devoted to the proof of this result.

Given $(X, G)$, endow $X$ with a $G$-invariant metric $g$, via the classical construction:
on the tangent space of a point $T_xX$,  build a $G$-invariant inner product
by taking the average over the stabilizer, and use the homogeneity of the space
to extend the metric to the whole manifold. The (simple) details on this construction
are left to the reader.

Choose a point $x \in X$, and consider $I(G_x)$, the identity component of the stabilizer of $G$.
Clearly, $G_x$ acts on $T_xX$, preserving the inner product on $T_xX$ given by the $G$-invariant metric
$g$. Therefore, $G_x$ must be a compact  subgroup of $O(3)$.

Since $I(G_x)$ is connected, there are not many choices for it: it can either be the trivial group,
$SO(2)$ or $SO(3)$. This subdivision play a crucial role in the classification

\beit
\item
{\bf Case 1:} If $I(G_x) = SO(3)$.

The manifold looks identical in every direction, so it can  be only $\EE^3$, $\HH^3$ or $S^3$. These
cases can be differentiated according to the sign of the sectional curvature, constant in every direction.

\item
{\bf Case 2:} If $I(G_x) = SO(2)$.

Let $L_x \subset  T_xX$ be the $1$-dimensional subspace fixed by $SO(2)$, and
let $P_x$ be the orthogonal complement (with respect to the $G$-invariant metric $g$) of
$L_x$ in $T_xX$.

Since $I(G_x)$ is a normal subgroup, both $L_x$ and $P_x$ are invariant under
the action of $G_x$. Furthermore, since $X$ is simply-connected,
we can choose a coherent orientation to let $L_x$ define a unit vector
field $v_x$ on $X$.

Note that the vector field does not need to be invariant under
the whole group $G$, but it will be invariant under some subgroup
$G_1$ of $G$ of order at most $2$.
Also the plane field defined by $P_x$ must be also invariant under $G_1$.
Therefore, both $v_x$ and $P_x$ descend to any manifold covered by $X$ and
with covering group contained in $G_1$.

Let $\phi_t(.)$ be the flow generated by $v_x$, and note that  it also descends
to the covered manifold. The plane field $P_x$ inherits an inner product
from $T_xX$, which is preserved by $\phi_t$.

To see this, take any vector $u \in P_x$, and set $||d\phi_t(u)|| = f(t)||u||$.
Since the vector field is preserved by the flow, the volume form will get multiplied by
a constant factor $f(t)$. Since the manifold is compact, $f(t) = 1$.

Now, if the plane field is integrable, $X$ must be either $S^2 \times \RR$,
$\HH^2 \times \RR$ or $\EE^2 \times \RR$. The latter will be disregarded,
since we are only looking for maximal geometries.

If $P_x$ is not integrable, then $X$ must be isometric to either $\tilde(SL(2, \RR))$,
$Nil$ or $S^3$, which will also be discarded for maximality.

\item
{\bf Case 3:} If $I(G_x) = G_e$ is trivial.

This means that $G_e$ acts transitively and freely on $X$, so we can
identify $X$ with the group $G_e$ itself. Since the geometry  $(X, G)$
was assumed to admit a compact quotient, it follows that $G_e$ has a
subgroup $H$ such that the quotient $G_e / H$ is compact.

\bede
A Lie Group $G$ is {\bf unimodular} if its left-invariant Haar measure
is also right-invariant.
\eede

If $G$ has a discrete subgroup $H$ such that $H \backslash G$ has a finite measure
inherited from the left invariant Haar  measure on $G$, then  $G$ must be unimodular.
Therefore, our $G_e$ must be unimodular.

Milnor \cite{Milnor} classified all simply  connected, unimodular Lie Groups:
$S^3$, $\EE^3$, $\tilde{SL(2, \RR)}$, $\tilde{\text{Isom}(\EE^2)}$, $Nil$ and  $Sol$.

Since we are requesting that $(X, G)$ is a maximal geometry,  the only possibility left is that
$X= Sol$. We must check that for any left-invariant metric on $Sol$, $I(G_x)$
is trivial. The argument we follow is provided by Milnor, in \cite{Milnor}.
He shows that a left-invariant metric on $Sol$
determines in a canonical way an orthonormal basis of the tangent space at the identity.
Such basis consists of eigenvectors of a naturally defined self-adjoint map
on the tangent space at  the identity, with associated eigenvalues
$\la_1 > 0$, $\la_2 = 0$ and $\la_3 < 0$.

Hence, any isometry of $Sol$ which fixes the identity must satisfy very specific
conditions on the eigenvectors. One can  then deduce that the
stabilizer of a point has order at most $8$, and if it has order exactly $8$,
it must be isomorphic to $D(4)$, which can be realized via the metric we described on
last section.

\eeit
 This completes the proof of Theorem \ref{t.Thurston}.

The proof of the following theorem, together with a very nice survey on
$3$-dimensional geometries, can be found in
\cite{PeterScott}.

\bethm
If $M$ is a closed $3$-manifold which admits a geometric structure
modelled in one of the eight possible model geometries, then it must be unique.
\eethm

\clearpage\mbox{}\clearpage

\chapter{Ricci Flow on Homogeneous Geometries}

This chapter will be devoted to the study of the behavior of Ricci flow
on homogeneous geometries. This will help the reader to build up some intuition for
the Ricci Flow, as in this case, the partial differential equation will be simplified into
a system of coupled ordinary differential equations due to homogeneity of the  space.
We will follow closely the approach chosen by \cite{Ben1}, as it exemplifies
very well the main character of this evolution equation.

We will restrict our attention to the case of the five model geometries that can be realized
as a pair $(G,  G)$, where $G$ is a simply-connected,  unimodular Lie Group.

\section{ Ricci Flow as a system of ODE's}

Let $G$ be any Lie Group, and let $\cal G$ be its Lie Algebra. The set of
left-invariant metrics on $G$ can be naturally identified with the set $S_n$ of
symmetric $n \times n$-matrices. For each metric, the Ricci
flow defines a path $g(t) \in S_n$, so we immediately see the reduction of the
flow to a system of $n(n+1)/2$ equations. However, by using a smart choice of moving
frames (the {\bf Milnor Frame}), we will be able to diagonalize the system, as follows.

First, we will equip $\cal G$, the Lie Algebra of a $3$-dimensional Lie Group $G$,
with a left-invariant moving frame $\{ F_i\}$, $i = 1, 2,3.$
We define structure equations by
\beeq
[F_i, F_j] = c_{ij}^k F_k,
\eeeq
and the {\bf adjoint representation} of $\cal G$
is the map $ad: {\cal G} \rightarrow gl(\cal G)$
defined by
\beeq
ad(V)(W)  = [V, W].
\eeeq

Once the orthonormal frame $\{ F_i\}$ is fixed,
we define an endomorphism ${\cal G} \rightarrow {\cal G}$
by
\beeq
F_1 \mapsto [F_2,F_3], \hspace{0.3cm}
F_2 \mapsto [F_3,F_1], \hspace{0.3cm}
F_3 \mapsto [F_1,F_2].
\eeeq
For our fixed basis, the matrix $C$ representing
this endomorphism has the structure constants as
its entries.

Our goal is to argue that there exists a choice of
frame that makes the matrix $C$ be diagonal.
Observe that if $G$ is unimodular, then
$tr(ad(V))$ = 0 for every $V \in \cal G$. It is a very
simple exercise to check that this condition implies that
the matrix $C$ is necessarily self-adjoint, and hence
there exists an orthogonal change of basis such that
the matrix $\hat C$ representing the endomorphism in the
new basis is orthogonal.

By reordering, and dropping the assumption that the new frame
must be {\bf orthonormal} (we will only need to keep it orthogonal), we
may assume that the matrix for this endomorphism is given by
$$
C = \begin{pmatrix}2 \la & 0 & 0 \cr 0 & 2 \mu & 0 \cr 0 & 0 & 2 \nu \end{pmatrix},
$$
where $\la \leq \mu \leq \nu$, and
$\la, \mu, \nu \in \{-1, 0, 1\}$. This choice of frame is called
a {\bf Milnor Frame}.

Now, let $\{ F_i \}$ be a Milnor frame for some left-invariant metric $g$.
Then, there exist  $A, B, C > 0 $ such that $g$ can be written as
$$
g = A \om^1 \otimes \om^1 +
B \om^2 \otimes \om^2 +
C \om^3 \otimes \om^3
$$
using the dual frame  $\{\om^i \}$.

In order to evolve the metric $g$ by the Ricci flow,
we will compute explicitly its curvature. Recall that the
{\bf Levi-Civita connection} of the metric $g$ is given by
\beeq
\nabla_X Y = \frac{1}{2}\left\{ [X, Y] - (ad X)^* Y - (ad Y)^* X \right\},
\eeeq
and once we know this expression, we can write the curvature tensor $R$ of the metric
$g$ as
\beeq
R(X, Y) Z  = \nabla_X \nabla_Y Z - \nabla_Y \nabla_X Z - \nabla_{[X,Y]}Z.
\eeeq

Now, note that if $\{ F_i\}$ is a Milnor frame, then it is a simple computation
to check that the map $ad^*$ is determined by
$$
ad^*  = \left( (ad F_i)^* F_j \right)^j_i
= \begin{pmatrix}
0 & 2 \la \frac{A}{C} F_3 & -2 \la \frac{A}{B} F_2 \cr
-2 \mu \frac{B}{C} F_3 & 0 & 2 \mu \frac{B}{A} F_1 \cr
2 \nu \frac{C}{B} F_2 & -2 \nu \frac{C}{A} F_2 & 0
\end{pmatrix}
$$

\begin{lemma}
For our choice of a Milnor frame, we have that
$$
\ip{R(F_k, F_i)F_j, F_k} = 0.
$$
\end{lemma}

The proof of this lemma is obtained immediately by
our explicit expressions for the Levi-Civita connection
and curvature.

This result tells us that a choice of Milnor frame allows us to identify globally
both $g$ and $Ric(g)$ with diagonal matrices. So the Ricci flow is reduced to a
system of only $3$ equations, instead of $6$, as expected.

Having completed the necessary background,  we will present the behavior of the flow of
homogeneous metrics with the $3$ possible isotropy groups.

\section{Ricci Flow of geometry with isotropy $S0(3)$}

This group includes $\RR^3, \HH^3$ and $S^3$. We will restrict our
attention to
$S^3$, to be identified with the Lie Group $SU(2)$ of complex $2 \times 2$ matrices
with determinant equal to $1$. The signature for a Milnor frame on
$SU(2)$ is given by $\la = \mu = \nu = -1$.

We will study the behavior of the Ricci Flow on $SU(2)$ with respect to a $1$-parameter
family of initial data exhibiting collapse (in the Gromov-Hausdorff sense) to a lower dimensional
manifold.

Consider the {\bf Hopf Fibration} $S^1 \mapsto S^3 \rightarrow S^2$ induced by the
projection $\pi : SU(2) \rightarrow \CC\PP^1$ given by $\pi(w,z) = [w,z]$.
In \cite{Berger}, Berger exhibit a collapse of $S^3$ to a round $S^2$ while
keeping the curvature bounded. Following \cite{Ben1}, we study here a family of
left-invariant initial metrics $\{g_\ve \}$, $\ve \in (0, 1]$,  on $SU(2)$. With
respect to a fixed Milnor Frame, we will write these metrics as
\beeq
g_\ve = \ve A \om^1 \otimes \om^1 +
B \om^2 \otimes \om^2 +
C \om^3 \otimes \om^3
\eeeq

A simple exercise is left to the reader: using the expression for the Riemann tensor in terms
of the structural constants, together with the definition of {\bf sectional curvature} of the plane
spanned by $F_i $ and $F_j$,
$$
K(F_i,F_j) = \ip{R(F_i, F_j)F_j, F_i},
$$
 you may check that the Ricci tensor of the metric $g_\ve$ is determined by
\begin{eqnarray}
\label{e.su2ric}
\Ric(F_1,F_1) & = & \frac{2}{BC}\left[ (\ve A)^2 - (B-C)^2 \right] \\
\Ric(F_2,F_2) & = & \frac{2}{\ve A C}\left[ (B^2 - (\ve A-C)^2 \right] \\
\Ric(F_3,F_3) & = & \frac{2}{\ve A B}\left[ C^2 - (\ve A -B)^2 \right]
\end{eqnarray}

The following proposition tells us that independently of the choice of a initial
homogeneous metric on $SU(2)$, the Ricci Flow will shrink to a round point in finite time.

\begin{prop}
For any $\ve \in (0,1]$ and any choice of initial data $A_0, B_0, C_0 > 0$,
the solution $g_\ve$ of \ref{e.RF} exists for a maximal finite time $T_\ve < \infty$,
becoming asymptotically round.
\end{prop}

\begin{proof}
The proof we outline here is provided in further detail in \cite{Ben1}.
It has two main steps: we first notice that (\ref{e.su2ric}) determine
the right-hand side of the system for the evolution of $A, B$ and $C$.
Due to the symmetry of the system, we can assume that
$\ve A_0 \leq C_0 \leq B_0$.

It is simple to check that
$$
\frac{d}{dt} \log (B - \ve A) = 4 \frac{C^2 - (B+ \ve A)^2}{\ve ABC},
$$
which shows that the relation $\ve A \leq C \leq B$ is preserved under the flow.
Furthermore, note that
$$
\ddt B \leq -8 + \frac{4\ve A}{C} \leq -4,
$$
so the solution can only exist on a finite time $T_\ve$.

The second step is to show that the metric becomes round as $t \rightarrow \infty$.
One can check that
$$
\ddt \frac{B - \ve A}{\ve A} \leq 0,
$$
so this quantity is bounded above by its initial value. Therefore, for $\de = \frac{B_0 - \ve A_0}{\ve A_0} $,
we have
$$
0 < B - \ve A \leq \de \ve A,
$$
completing the proof of the proposition.
\end{proof}

In contrast with the other model geometries, the Ricci flow of
a homogeneous metric will avoid collapse. We will refer to
\cite{Ben1} for the proof that in fact, the maximal times $T_\ve$
can be chosen uniformly for all $\ve \in (0, 1]$.

For the result about collapsing, we define the quantities
$E = B + C$, and $F = (B - C)/\ve$.
\begin{prop}
If $g_\ve(t)$,  a family of metrics parametrized
by
$\ve \in (0, 1]$, satisfies
$$
\lim_{\ve \rightarrow 0} F(0,\ve) > 0,
$$
then for all $t \in (0, T_{\text max})$,
$$
\lim_{\ve \rightarrow 0} F(t,\ve) = 0.
$$
\end{prop}

This proposition shows that there is a jump discontinuity at $t = 0$.

\begin{proof}
We just present here an outline. The evolution equation for
$F$ is given by
$$
\ddt \log F = \frac{4}{ABC} \left(  \ve A^2 - \frac{E^2}{\ve} \right).
$$
So, the proof follows from noticing that
$$
\frac{4}{ABC} \left(  \ve A^2 - \frac{E^2}{\ve} \right) \rightarrow \infty
\hspace{1cm} \text{as} \hspace{1cm} \ve \rightarrow 0.
$$
\end{proof}

\section{Ricci Flow of geometry with isotropy $SO(2)$}

The geometries contained in this group are
$S^2  \times \RR$, $\HH^2 \times \RR$, $Nil$ and $\tilde{SL(2, \RR)}$.
The former has a distinct behavior amongst the others: under the Ricci flow,
the $\RR$-factor
remains fixed, while the $S^2$-factor shrinks. Therefore, the solution
becomes singular in finite time and converge to a $1$-dimensional manifold.

For the remaining geometries, some of its directions get expanded by the flow,
while some others will converge to a finite value.
For this case, we will study in detail the Lie Group $Nil$, as all the geometric objects
are simple to compute, and it captures the main behavior of the Ricci flow on $SO(2)$ isotropy.

Let $G$ denote the $3$-dimensional group of matrices of the form
$$
\begin{pmatrix}
1 & x & z \cr
0 & 1 & y \cr
0 & 0 & 1
\end{pmatrix},
\hspace{1cm}
x, y, z \in \RR.
$$

The signature of a Milnor frame for $G$ is $\la = -1$,  $\mu = \nu = 0$.
As before, we can write any left-invariant metric $g$ on $G$
as
$$
g = A \om^1 \otimes \om^1 +
B \om^2 \otimes \om^2 +
C \om^3 \otimes \om^3
$$
with respect to a Milnor frame of 1-forms $\{\om_i \}$.

Also, the reader can check that the Ricci tensor of the metric $g$ will be given by
$$
\Ric = \frac{2A^2}{BC} \om^1 \otimes \om^1
-\frac{2A}{C} \om^2 \otimes \om^2
-\frac{2A}{B} \om^3 \otimes \om^3.
$$
Hence, the evolution equations for $A, B$ and $ C$ are given by
$$
\ddt A  = \frac{-4A^2}{BC}; \hspace{0.5cm}
\ddt B  = \frac{4A}{C}; \hspace{0.5cm}
\ddt C  = \frac{4A}{B}.
$$

We are lucky enough to have explicit solutions to this system,
given by
\begin{eqnarray}
A(t) &=& A_0^{2/3} B_0^{1/3} C_0^{1/3}\left( 12t + B_0 C_0/A_0\right)^{-1/3} \\
B(t) &=& A_0^{1/3} B_0^{2/3} C_0^{-1/3}\left( 12t + B_0 C_0/A_0\right)^{1/3} \\
C(t) &=& A_0^{1/3} B_0^{-1/3} C_0^{2/3}\left( 12t + B_0 C_0/A_0\right)^{1/3}
\end{eqnarray}

So, we can see explicitly that, for the $Nil$ geometry,
the Ricci flow will expand two of the directions determined
by the Lie algebra, while the other will converge to a fixed value.
As a corollary, we see that any compact manifold with $Nil$ geometry
will converge, via the normalized Ricci flow, to  $\RR^2$ endowed with the
flat metric.

\section{Ricci Flow of geometry with trivial isotropy }

The only geometry in this category is $Sol$, and the signature of
a Milnor frame on this group is $\la = -1$, $\mu = 0$ and $\nu = 1$.
Using this quantities, we obtain the following evolution equations for the
coefficients of a left-invariant metric $g$ (with respect to a Milnor frame):
\begin{eqnarray}
\ddt A & = & 4 \frac{C^2 - A^2}{BC} \\
\ddt B & = & 4 \frac{C^2 + A^2}{AC} \\
\ddt C & = & 4 \frac{-C^2 + A^2}{AB}.
\end{eqnarray}

Note that $C \ddt A + A \ddt C  = 0$, so the quantity $AC$ is preserved. So, we can
define $G = A/C$, and rewrite the system above as
\begin{eqnarray}
\ddt B & = & 8 + 4 \frac{1 + G^2}{G} \\
\ddt G & = & 8\frac{1 - G^2}{B}.
\end{eqnarray}

We have that $\ddt B \geq 16$, so $B(t) \rightarrow \infty$.
It is also true that $G(t) \rightarrow 1$ as $t \rightarrow \infty$, and in fact,
both $A$ and $C$ will converge to the same value.

Therefore, our conclusion is that the solution $g(t)$
to the Ricci flow exists for all times. Furthermore, if we consider the
normalized Ricci flow on a compact $3$ manifold modeled on $Sol$, we will
observe convergence to $\RR$ in the Gromov-Hausdorff sense.

\chapter{Ricci Flow on Surfaces}

In this chapter, we will explain the behavior of the Ricci flow
on compact surfaces. In dimension $2$, the Ricci flow equation is simply given
by
\beeq
\label{e.RFsurface}
\ddt g(t) = - R g.
\eeeq

In order to keep the total area $A$ of the surface constant, we will introduce an
extra constant in the equation, obtaining the {\bf normalized Ricci flow for surfaces}:
\beeq
\label{e.NRFsurface}
\ddt g(t) = (r - R) g,
\eeeq
where $r = \frac{4 \pi \chi(M)}{A}$ \footnote{This is a direct application of Gauss-Bonnet theorem.} is the average of the scalar curvature on the surface.

Note that, at each point, the rate of change is a multiple of the metric, so
we will be flowing the metric inside its conformal class. In fact, this equation makes sense
even in higher dimensions, and is known as the {\bf Yamabe flow}. The Yamabe
flow is the gradient flow
related to the so-called
{\bf  Yamabe problem}: fixed a conformal class of a metric on a compact
manifold, as well as its volume, is there a metric of
 constant scalar curvature in the fixed
conformal class?

In a sense, the $2$-dimensional Ricci flow resembles much more of the character of the Yamabe
flow. Nevertheless, we hope that the study of the Ricci flow on surfaces will be a good introduction to
the general and far more complicated case of $3$ dimensions.

The goal of this chapter is to prove the following results.

\begin{thm}{\bf (Hamilton)}
For any initial data, the solution to the normalized Ricci flow equation
(\ref{e.NRFsurface}) exists for all times.
\end{thm}

\begin{thm}{\bf (Hamilton)}
If $r \leq 0$, the metric converges to a constant curvature metric.
\end{thm}

\begin{thm}{\bf (Hamilton)}
If $R > 0$, the metric converges to a constant curvature metric.
\end{thm}

An important observation is that from the work of Hamilton, one
cannot deduce directly the Uniformization Theorem for surfaces. Even though it
treats successfully the cases where $\chi(M) < 0$ and $\chi(M) = 0$, it uses the fact that
there exists a metric of constant positive curvature on $S^2$.

Recently, however, Chen, Lu and Tian \cite{XXChen-Lu-Tian} provided a pure Ricci flow proof of the
Uniformization conjecture. We will come back to this at the end of the chapter.

\section{Some estimates}

In order to study the Ricci flow on surfaces, we start by developing the
evolution equation for the scalar curvature $R(t)$ of the metric $g(t)$. As
previously noted, the metric $g(t)$ will be always a conformal deformation of the
initial metric $g_0$. So, it is useful to develop the relation between the scalar
curvatures of two conformal metrics.

\begin{lemma}
If $g $ and $h = e^{2u}g$ are two conformal metrics, then
\beeq
\label{e.conformalscalarrelation}
R_h = e^{-2u}(-2\De_g u + R_g)
\eeeq
\end{lemma}

Before we present the proof of the lemma above, let us recall  the technique of {\bf moving frames}.
 Let $\{\om_1, \om_2 \}$ be a
coframe field, chosen to be orthonormal with respect to the metric $g$. So,
we can write the metric $g$ as
$$
g = \om^1 \otimes \om^1 + \om^2 \otimes \om^2.
$$

Let $\{f_1, f_2\}$ be the dual frame ({\em i.e.}, $f_i \om^j = \de^j_i$).

We define the {\bf connection 1-forms} $\om^j_i$ by the following relation:
let $X$ be a vector field. Then,
$$
\nabla_X f_i = \om^j_i f_j.
$$

Using this, we can write the Cartan structure equations for a surface as
\begin{eqnarray}
\label{e.Cartan}
d\om^1 & = & \om^2 \wedge \om^1_2; \\
d\om^2 & = & \om^1 \wedge \om^2_1; \\
\Om^1_2 &=& d \om^1_2.
\end{eqnarray}

\begin{proof}
As before, let  $\{ f_1, f_2\}$ be a orthonormal frame for $g$,
dual to $\{\om_1, \om_2 \}$. Then
$$
e_1 = e^{-u} f_1  \hspace{2cm} e_2 = e^{-u}f_2
$$
is  an orthonormal moving frame for $h$. Call
$$
\eta_1 = e^{u} \om_1  \hspace{2cm} \eta_2 = e^{u}\om_2
$$
the coframe for $h$. The proof follows from:
\beit
\item
Writing $d\eta^1$ and $d\eta^2$ in terms of $u$, $\om^1$ and $\om^2$;
\item
Computing
$$
\eta^1_2 = d\eta^1(e_2,e_1)\eta^1 +  d\eta^2(e_2,e_1)\eta^2,
$$
and applying to the second structure equation of Cartan (\ref{e.Cartan});
\item
Noticing that, in an orthonormal frame, $\De_g u = \nabla^2_{f_1,f_1} u + \nabla^2_{f_2,f_2} u $,
which yields
$$
(\Om_h)^1_2 = (\Om_g)^1_2 - \De_g u \om^1 \wedge \om^1;
$$
\item
Finally, recalling that
$$
R_h = 2(\Om_h)^1_2 = e^{-2u} (R_g - 2 \De_g u).
$$
\eeit

\end{proof}

Now, we study the evolution of the curvature $R(t)$ under the Ricci flow.

\begin{lemma}
If $g(t)$ is a $1$-parameter family of metrics on a surface $M^2$ such that
$$
\ddt g = f g
$$
for some function $f$, then
$$
\ddt R = -\De f - Rf.
$$
In particular, if $f = r-R$ (the normalized Ricci flow equation), then
\beeq
\label{e.Rsurface}
\ddt R = \De R + R(R-r).
\eeeq
\end{lemma}

\begin{proof}
We have already seen the relation between the scalar curvatures
of two conformal metrics $g$ and $h = e^{-2u}g$ is given by (\ref{e.conformalscalarrelation}).
So, in that case, if  $\ddt g = f g$, then $\ddt u = f$. Differentiating
(\ref{e.conformalscalarrelation}) with respect to time, we obtain
$$
\ddt R_h = - (\ddt u) e^{-u}(-\De_g u + R_g) - e^{-u} \De_g(\ddt u) =  -\De_h f - R_h f.
$$

\end{proof}

\subsection{Scalar Maximum Principle}

At this point, we will start getting into the spirit of geometric analysis. We would like to
establish some bounds for the scalar curvature of a metric running under the Ricci flow.
Our strategy is to compare a solution of the PDE
$$
\ddt R = \De R + R(R-r)
$$
with the solution of the associated ODE to this problem, obtained by ignoring the Laplacian
term on the equation above. In order to do so, we need to use a version of the scalar maximum
principle for a heat equation with non-linear reaction term.

For completeness, we state here the version of the Maximum Principle that we will use.
Consider the following (more general) equation for the evolution of a scalar quantity $v$:
\beeq
\label{e.scalarmaxprinc}
\ddt v = \De_{g(t)}v + F(v),
\eeeq
where $g(t)$ is a $1$-parameter family of metrics, and $F: \RR \rightarrow \RR$ is a locally
Lipschitz function.

The function $u$ is said to be a {\bf supersolution} to (\ref{e.scalarmaxprinc})
if
$$
\ddt u \geq \De_{g(t)}u + F(u)
$$
and a {\bf subsolution} if
$$
\ddt u \geq \De_{g(t)}u + F(u).
$$

\begin{thm}
{\bf (Scalar Maximum Principle) }
Let $u$ be a $C^2$ supersolution to (\ref{e.scalarmaxprinc}) on a closed
manifold $M$.
Suppose that there exists $C_1 > 0$ such that $u(x,0) \geq C_1$ for all $x$ in $M$,
and let $\phi_1$ be the solution to the ODE
$$
\ddt \phi_1 = F(\phi_1),
$$
with initial condition $\phi_1(0) = C_1$.

Then, $u(x, t) \geq \phi_1(t)$ for all $x \in M$, and for all $t$ such that
$\phi_1$ exists.

Analogously, let $u$ be a subsolution to (\ref{e.scalarmaxprinc}).
Assume that there exists $C_2 > 0$ such that $u(x,0) \leq C_2$ for all $x$ in $M$,
and let $\phi_2$ be the solution to the ODE
$$
\ddt \phi_2 = F(\phi_2),
$$
with initial condition $\phi_2(0) = C_2$.

Then, $u(x, t) \leq \phi_2(t)$ for all $x \in M$, and for all $t$ such that
$\phi_2$ exists.

\end{thm}

\begin{proof}
We will only proof the first part of the theorem, as the proof of the second part is
analogous. For simplicity, let $\phi_1 = \phi$, $C_1 = C$.
We have
$$
\ddt (u - \phi) \geq \De(u - \phi) + F(u) - F(\phi),
$$
and $(u - \phi)|_{t = 0} \geq 0$.

\noindent
{\bf Claim:} $u - \phi \geq 0$ for all times where $\phi$ is defined.

In order to prove the claim, fix $\tau \in (0, T)$. Since $M$ is compact,
there exists $K_\tau$ such that
both $|u(x, t)|$ and $|\phi(t)|$ are bounded above by $K_\tau$,
for all $x \in M$, for all $t \in [0,\tau]$.

Let $L_\tau$ be the Lipschitz constant for $F$ in the $\tau$-slice. Then
$$
\ddt (u - \phi) \geq \De(u - \phi) \pm L_\tau (u - \phi).
$$

Define
$ J = e^{K_\tau t}(u - \phi)$. An easy computation gives
$$
\ddt J \geq \De J.
$$
Since $J(0) \geq 0$, the result follows from the maximum principle
for the heat equation, as follows.

We write $H = J + \ve t + \ve$ for some $\ve > 0$. The heat equation for
$J$ implies that
$$
\ddt H \geq \De H + \ve.
$$

From our old Calculus lessons, we remember that, if $(x_0, t_0)$ is the point where
the minimum of $H$ is attained, {\em i.e.},
$$
H(x_0, t_0) = \min\{H(x,t); x \in M, t \in [0, t_0]\},
$$
then
\beeq
\label{e.calculus}
\ddt H (x_0, t_0) \leq 0;
\hspace{.5cm}
\nabla H (x_0, t_0) = 0;
\hspace{.5cm}
\De H (x_0, t_0) \geq 0.
\eeeq

Now, suppose by contradiction  that our $H = J + \ve t + \ve$ is non-positive at some point $(x_1, t_1)$.
Since $H(.,0) > 0$, then there exist a first time $t_0$ and a point $x_0$ such that
$H(x_0, t_0) = 0$, which implies that $J(x_0, t_0) < 0$.

Recall that
$H(x_0, t_0) = \min \{H(x,t); x \in M, t \in [0, t_0]\}$. Therefore, (\ref{e.calculus}) yields
$$
0 \geq \ddt H (x_0, t_0) \geq \De H (x_0, t_0) + \ve \geq \ve > 0,
$$
which is a contradiction.

This completes the proof of the theorem.
\end{proof}

\subsection{Back to the evolution of $R$}

As a direct application of the maximum principle to  (\ref{e.Rsurface}),
we see that, if at $t=0$ we have $R \geq 0$, then this relation holds for
all times where the solution exists. Similarly, the condition $R \leq 0$ is also
preserved for the normalized Ricci flow on surfaces. An important observation is that
for higher dimensions, these quantities are not preserved under Ricci flow.

In fact, if $R(0) \leq 0$, the maximum principle tells us even more.
\begin{prop}
If  there exist $C, \ve > 0$ such that
$-C \leq R(0) \leq -\ve < 0$, then
$$
r e^{\ve t} \leq r - R \leq C e^{rt}
$$
so $R$ approaches $r$ exponentially.
\end{prop}

\begin{proof}
For a fixed $t$, let $\rho(t) = \min_{x \in M} R(x,t)$
and $\varrho(t) = \max_{x \in M} R(x,t)$. Then,
$\varrho$ satisfies
$$
\ddt \varrho \leq \varrho(\varrho - r) \leq -\ve(\varrho - r)
$$
and the minimum  $\rho$ satisfies
$$
\ddt \rho \geq \rho(\rho - r) \geq r(\rho - r),
$$
which clearly imply the claim.
\end{proof}

\begin{cor}
\label{c.chineg}
On a compact surface, if we start the Ricci flow with an initial
metric $g(0)$ whose scalar curvature is negative, then the solution exists
for all times, and converges exponentially to a metric of constant curvature.
\end{cor}

This is almost the result we seek in the case of a surface with
$\chi(M) < 0$. The only remaining piece is to show that if $r < 0$, then
the scalar curvature $R(t)$
will eventually become negative at some time. If that were true, we just
restart the flow at that time, and the metric will flow exponentially to
the metric with curvature equal to $r = \frac{4\pi \chi(M)}{A}$.
We will be back to that soon.

Sadly, the situation if far more complicated
for $\chi(M) >0$. But we still have uniform lower bounds that are also consequence
of the maximum principle.

\begin{prop}
Let $\rho(t) = \inf_{x \in M} R(x,t)$. Then
\beit
\item
If $r > 0$, then
$$
R - r \geq \frac{r}{1 - ( 1 - \frac{r}{\rho(0)} )e^{rt}} - r \geq (\rho(0)- r)e^{rt}.
$$
\item
If $r = 0$, then
$$
R-r \geq \frac{\rho(0)}{1 - \rho(0)t}.
$$
\item
If $r < 0$ and $\rho(0) < 0$, then
$$
R \geq \frac{r}{1 - ( 1 - \frac{r}{\rho(0)} )e^{rt}} \geq \rho(0) e^{-rt}.
$$
\eeit
\end{prop}

This proposition follows from the maximum principle, that allows us to compare
the solution of (\ref{e.Rsurface}) with the solution to the associated scalar ODE:
$$
\ddt R = R(r - R).
$$

It seems that we juiced out all we could from the maximum principle for this equation. In order
to obtain nice upper bounds for $R$ evolving under Ricci flow, we need to use a smarter tool.

\section{Ricci Solitons on surfaces}

In this section, a new tool is developed in order to produce the expected upper bounds.

\begin{defin}
Let $g(t)$ be a solution of the normalized Ricci flow (\ref{e.NRFsurface}) on a surface $M$.
We say that $g(t)$ is a {\bf self-similar solution} if there exists a $1$-parameter family
$\varphi (t)$ of conformal diffeomorphisms such that
\beeq
\label{e.self-similar}
g(t) = [\varphi(t)]^* g(0).
\eeeq
\end{defin}

Equation (\ref{e.self-similar}) implies that
$$
\ddt g(t) = {\cal L}_X g,
$$
where $X$ is the vector field generated by $\vp (t)$ and ${\cal L}
_X$ denotes the Lie derivative of the metric
in the direction of $X$. More is true: if $g(t)$ is a solution to the \NRF, then
\beeq
\label{e.solitonSurface}
(R - r)g_{ij} = \nabla_i X_j + \nabla_j X_i.
\eeeq
We shall refer to (\ref{e.solitonSurface}) as the {\bf Ricci soliton equation}.

Also, if there exists a function $f$ such that $\nabla f = X$, then we have a so-called
{\bf gradient Ricci soliton}, which will satisfy the equation
\beeq
\label{e.gradientsoliton}
(R - r)g_{ij} = \nabla_i \nabla_j f.
\eeeq
The function $f$ will be referred to as the {\bf Ricci potential}.

Tracing the equation above, we see that the potential $f$ must satisfy
$$
\De f = R - r.
$$
This equation is solvable on a compact manifold, since
$$
\int_M (R - r) d\Vol = 0
.$$

A gradient Ricci soliton is a very special solution of the \NRF, and as so,
we expect that some quantities related to it will be preserved along the flow.

\begin{lemma}
For a gradient Ricci soliton, the expression
$$
R + |\nabla f|^2 + r f
$$
is only a function of time.
\end{lemma}

\begin{proof}
Let $M  = \nabla \nabla f - 1/2 \De f g$ be the trace-free part of the
hessian of $f$. Clearly, $M = 0$ is a necessary and sufficient condition
for a gradient Ricci soliton.

Computing the divergence of $M$, we have
$$
(\text{div} M)_i = \nabla^j M_{ij} = \frac{1}{2}(R \nabla_i f + \nabla_i R).
$$
Therefore, for a gradient Ricci soliton,
\begin{multline}
0 = \nabla_i R +  R \nabla_i f =
\nabla_i R +  (R-r) \nabla_i f + r\nabla_i f = \\
= \nabla_i ( R + |\nabla f|^2 + rf),
\end{multline}
which completes the proof.

\end{proof}

In fact, we can choose our potential to satisfy an even nicer equation.
\begin{lemma}
Let $f_0$ be a potential. Then, there exists $c = c(t)$
(that only depends on time) such that the new potential
$f = f_0 + c$
satisfies
\beeq
\label{e.potential}
\ddt f = \De f  + r f.
\eeeq
\end{lemma}

The proof of this lemma follows from noticing that
$\ddt \De = (R - r) \De$, and recalling that the only
harmonic functions on a compact surface are constants.

Applying the maximum principle to (\ref{e.potential}),
we see that there exists a constant $K$ such that $|f| \leq C^{rt}$.
We still need to work a bit further to extract an upper bound for $R$.
For that, we define
\beeq
h = \De f + |\nabla f|^2 = (R - r)f + |\nabla f|^2.
\eeeq

\begin{lemma}
The evolution for $h$ is given by
$$
\ddt h = \De h - 2 |M|^2 + rh,
 $$
where $M$ is the trace-free part of the Hessian of $f$.

Therefore, if $h \leq C$ at time zero, then $h(t) \leq C^{rt}$ for all $t$.
\end{lemma}

\begin{proof}
From the evolution of $R$, we have
$$
\ddt R = \De (R-r) + R(r-R) = \De(R - r) + (\De f)^2 + r(R - r).
$$
Also,
$$
\ddt |\gr f|^2 = \ddt (g^{ij} \gr_i \gr_j f) = \De (|\gr f|^2) - 2|\gr \gr f|^2 + r |\gr f|^2,
$$
which implies the claim.
\end{proof}

Observing that $R = h + r - |\nabla f|^2$, we obtain our desired upper bound:

\bethm
\label{t.surfacebound}
{\bf (Hamilton)}
For any initial metric on a compact surface, there exists a constant
$C$ such that
$$
-C \leq R \leq Ce^{rt} + r.
$$

Therefore, the Ricci flow equation has a solution defined for all times.
\eethm

For the $r < 0$ case, we complete the proof of the main result of this chapter,
as the theorem above implies that eventually $R$ will become negative. Combining this
with Corollary \ref{c.chineg}, we obtain
\begin{cor}
If $r < 0$ (a purely topological condition), then for any choice of initial metric,
the solution to the \RF exists for all times, and converges to a metric with constant
negative curvature.
\end{cor}

\section{The case $r = 0$}

From Theorem \ref{t.surfacebound}, we know that the solution exists for all times,
and that $R$ is bounded from above and below. It remains to show that
the solution actually converges to a  flat metric.

Recall that the \RF on surfaces evolves inside a conformal class. So, consider
$g = e^{-2u}h$, two conformal metrics. The relation between the scalar curvatures
is given by
$$
R_g = e^{2u} (R_h - 2\De_h u).
$$

Hence, up to replacing the starting metric by a conformal factor $u$ such that
$\De u = R$\footnote{This equation is solvable, as the average scalar curvature is zero},
we may assume that $h$ is the flat metric, and study the evolution of the conformal factor $u$.

Note that
$$
-2\ddt u e^{-2u}h = \ddt g = R_g g = R_g e^{-2u}h = -2 \De_h u h,
$$
which shows that the conformal factor $u$ evolves by
\beeq
\label{e.conformalfactor}
\ddt u = e^{2u}\De_h u.
\eeeq

The maximum principle allows us to conclude the following.
\begin{cor}
There exists a constant $C \geq 0$ such that
$$-C \leq u(t)\leq C
$$
for all time $t$.
\end{cor}

This corollary has important consequences: all the metrics $g(t)$ are uniformly
equivalent, as well as the diameter and Sobolev constant.

From now on, our goal is to prove exponential decay for the scalar curvature $R(t)$.
The strategy, following \cite{H1}, is estimate the $L^2$-norms of $R, \gr R$ and $\gr^2 R$,
and use the Sobolev embedding theorem to obtain the expected decay.

In what follows, we will drop the subscript $h$ on the expressions for the geometric objects
related to the (flat) metric $h$.
We have
\beeq
\label{e.above1}
\frac{d}{dt} \int |\gr u|^2 d\mu = 2 \int \ip{\gr u, \gr(e^{2u} \De u)} d\mu = -2 \int e^{2u}\De u d\mu,
\eeeq
which yields, together with the relation
$$
\int(\De u )^2 d\mu \geq c \int |\gr u|^2 d\mu,
$$
$$
\frac{d}{dt} \int |\gr u|^2 d\mu + c \int |\gr u|^2 d\mu \leq 0.
$$

Thinking of this integral as a function $A(t)$, the relation above simply tells us that
$\frac{d}{dt} A \leq -c A$. This implies the exponential decay on the $L^2$-norm of
$\gr u$: for some $C > 0$, we have
$$
\int(\gr u )^2 d\mu \geq C e^{-ct}.
$$

Now, integrating (\ref{e.above1}) with respect to $t$, we get
$$
\int_T^\infty \int e^{2u}(\De u)^2 d \mu \leq \int(\gr u )^2 d\mu,
$$
and since $u$ is uniformly bounded,
$$
\int_T^\infty \int R^2 d\mu \leq \int_T^\infty \int e^{2u}(\De u)^2 d \mu \leq Ce^{-ct}.
$$

This relation tells us that there exist a point $\xi $ in every interval $[T, T+1]$ such that
$$
\int R^2 d\mu \leq Ce^{-c\xi}.
$$
Also, from the evolution equation of $R$,
$$
\ddt R^2 = 2R(R^2 + \De R ),
$$
and hence
\beeq
\label{e.above2}
\frac{d}{dt} \int R^2 d\mu  \leq \int R^3 d\mu \leq C \int R^2 d\mu,
\eeeq
where the last inequality comes from the fact that $R$ is bounded and the surface is compact.

Let $q(t) = \int R^2 d\mu $. We learned so far that in every interval $[T, T+1]$,
$q(\xi) \leq Ce^{-c\xi}$, and that $\frac{d}{dt}q \leq C q$. Therefore, for $t \in [T, T+1]$
\beeq
q(t) = q(\xi) + \int^t_\xi \frac{d}{dt}q
\leq Ce^{-c\xi} + C\int_T^\infty q \leq Ce^{-ct},
\eeeq
which shows the exponential decay
$$
\int R^2 d\mu \leq C e^{-ct}.
$$

Now, we proceed with the estimates for the gradient of $R$. Integrating
(\ref{e.above2}) with respect to time, and noting that any $L^p$-norm of
$R$ goes to zero exponentially, we get
$$
\int_T^\infty \int |\gr R|^2 d\mu \leq Ce^{-cT},
$$
which again shows that
$$
\int |\gr R|^2 d\mu \leq Ce^{-c\xi}
$$
for some $\xi$ in every interval $[T, T+1]$.

Integration by parts yields
$$
\frac{d}{dt} \int |\gr R|^2 d\mu + 2 \int (\De R)^2 d\mu \leq -2 \int R^2\De R d\mu ,
$$
and Cauchy-Schwartz gives that the right-hand side is bounded by
$$
-2 \int R^2\De R d\mu \leq \int R^4 d\mu + \int(\De R)^2d\mu.
$$
Therefore,
$$
\frac{d}{dt} \int |\gr R|^2 d\mu +  \int (\De R)^2 d\mu \leq \int R^4 d\mu
$$
Since the right-hand side of the equation above is exponentially small, this gives us two pieces
of information: firstly,
$$
\frac{d}{dt} \int |\gr R|^2 d\mu \leq Ce^{-ct},
$$
and secondly,
$$
\int_T^\infty \int (\De R)^2 d\mu \leq Ce^{-cT}.
$$

This last inequality tells us that we can play the same game as before, redefining the
quantity $q$ to be $q(t) =  \int (\De R)^2 d\mu$, and obtaining the desired exponential decay
on the $L^2$-norm of the Laplacian of $R$.

The bound on $\gr^2 R$ follows from the previous bounds, and the {\bf Bochner identity}
for the case of  a flat metric
$$
\De(\frac{1}{2} |\gr^2 R|^2) = |\gr^2 R|^2 + \ip{\gr R, \gr(\De R)}.
$$
imply that\footnote{Note that also an integration by parts is needed.}
$$
\int |\gr^2 R|^2 d\mu = \int(\De R )^2 d\mu - \frac{1}{2}R |\gr R|^2 d\mu.
$$

Therefore, with the $L^2$-norms of $R, \gr R $ and $\gr^2 R$ in hand, Sobolev's
embedding tells us that the maximum of $|R|$ goes to zero exponentially, which
completes our proof of the Uniformization Theorem in the case $\chi(M) = 0$.

\section{The case $R > 0$}

\subsection{Hamilton's Harnack inequality}

Following the seminal paper \cite{H1}, we begin by deriving a generalization,
by Hamilton, of the Li-Yau Harnack inequality (cf. \cite{Li-Yau}). For completeness,
we state here the classical Harnack inequality.

\bethm
Let $M$ be a compact manifold of dimension $n$ with a  fixed metric with non-negative
Ricci curvature. Let $f>0$ be a solution of
$$
\ddt f = \De f
$$
on $(0, T)$. Then, for any two points $(\xi, \tau)$ and $(X, T)$
in space-time with $0 < \tau < T$,
$$
\tau^{n/2} f(\xi, \tau) \leq e^{\De/4} T^{n/2}f(X, T),
$$
where $\De = \frac{d^2(\xi, X)}{T-t}$, and $d$ is the distance along the shortest geodesic.
\eethm

The idea of the proof, in very rough words,  is to study the evolution equation of the quantity $L = \log f$,
and apply the maximum principle for $Q = \De L$.

For the case of the \RF on surfaces, we need a better version of the Harnack inequality,
as the metrics on the manifold are varying. Hamilton's idea is to consider a new
definition for $\De$, as follows.

\begin{defin}
\label{d.delta}
Let $g(t)$ a family of metrics on a manifold $M$. Define
\beeq
\De((\xi, \tau)(X, T)) = \inf_\ga \int^T_\tau \frac{ds}{dt}^2 dt,
\eeeq
where the infimum is taken over all paths joining $(\xi, \tau)$ to $(X, T)$.
\end{defin}

Note that this definition coincides with the previous one in the case of a fixed metric.
Also, if there are two metrics $h$ and $G$, independent of time, with distances
$d_h(\xi, X)$ and $d_G(\xi, X)$, then
$$
\frac{d_h^2(\xi, X)}{T-t} \leq  \De((\xi, \tau)(X, T))   \leq \frac{d_G^2(\xi, X)}{T-t}
$$
whenever $h(x) \leq g(x, t) \leq G(x)$.

\bigskip

Now, we state Hamilton's Li-Yau Harnack inequality.

\bethm
Let $g$ be a solution of \RF on a compact surface, with $R > 0$ for
$0<\tau<T$.
Then, for any two points $(\xi, \tau)$ and $(X, T)$
in space-time with $0 < \tau < T$,
$$
(e^{r\tau} - 1) R(\xi, \tau) \leq e^{\De/4} (e^{rT} - 1) R(X, T),
$$
where $\De$ is as in Definition~\ref{d.delta}.
\eethm

\begin{proof}
Let $L = \log R $. Then,
$$
\ddt L = \De L  + |\gr R|^2 + R-r.
$$
Consider the quantity $Q = \ddt L - |\gr R|^2$. The evolution equation
for $Q$ is given by
\begin{multline}
\ddt Q = \De \left( \ddt L  \right) + \ddt(\De) L + \ddt R = \\
= \De(Q) + \De(|\gr L|^2) + (R - r)\De L + R \ddt L  =  \\
=\De(Q) + 2(|\gr^2 L| + \ip{\De \gr L, \gr L}) + (R - r)\De L + R (\De L  + |\gr R|^2 + R-r)
= \\=
\De Q + 2 |\gr \gr L|^2 +  \ip{\gr \De L, \gr L} + 2R|\gr L|^2 +
(2R -r) \De L + R(R-r)
= \\ =
\De Q + 2\ip{\gr L,  \gr Q} + 2|\gr \gr L - \frac{1}{2} (R - r)g|^2 + r Q
 \\
 \geq
\De Q + 2\ip{\gr L,  \gr Q} + Q^2+ r Q,
\end{multline}
where the inequality follows from
$$
Q^2 \leq 2|\gr \gr L - \frac{1}{2} (R - r)g|^2.
$$

Once again, the maximum principle applied for $Q$ allows us to
compare $Q$ with the solution of the associated ODE, giving
$$
Q \leq \frac{-re^{rt}}{e^{rt} - 1}.
$$

Now, let $\ga$ be any path joining two points
$(\xi, \tau)$ and $(X, T)$ in space-time. Now,
we just compute
\begin{eqnarray*}
L(X, T) - L (\xi, \tau) & = & \int_\tau^T \frac{d}{dt}L dt \\
& \geq &
\int_\tau^T \left[ |\gr L|^2 - \frac{-re^{rt}}{e^{rt} - 1} + \frac{\partial L}{\partial s} \frac{ds}{dt}\right] dt \\
& \geq &
-\log\left( \frac{e^{rT -1}}{e^{r\tau -1}} - \frac{1}{4} \int_\tau^T \frac{ds}{dt}^2 dt \right).
\end{eqnarray*}
The proof follows from taking exponentials, and noting that the infimum of the last integral
over all paths is the definition of $\De$.
\end{proof}

\subsection{Entropy estimate}

Another step in developing the behavior at infinite time of $R$ is the following result.

\bethm
Let $R$ be the scalar curvature of the solution of the \RF on a surface, with $R>0$.
Then,
$$
\int R \log R
$$
is decreasing, as a function of time.
\eethm

\begin{proof}
Following Hamilton, we consider
$$
Z = \frac{\int QR d\mu}{\int R d\mu}.
$$
Then, $Z$ satisfies
$$
\frac{dZ}{dt} \geq Z^2 + rZ.
$$

If $Z$ would become positive, it would blow up at a finite time,
contradicting the long-time existence result for the \RF on surfaces.
Therefore, {\bf $Z \leq 0$}.

Also,
$$
Q = \frac{\De R}{R} - \frac{|\gr R|^2}{R^2} + (R - r).
$$

Then, if $R > 0$, we get that
$$
\int(R - r)^2 d\mu \leq \int \frac{|\gr R|^2}{R}.
$$

To complete the proof, just observe that
\begin{multline}
\frac{d}{dt} \int R\log R d\mu  = \int \left[\frac{dR}{dt}\log R d\mu +
 \frac{dR}{dt}  +  R(R - r)\log R \right]d\mu  =\\ =
 \int \left[\De R \log R + \De R + R(r-R)\right] d\mu = \\ =
 \int (R - r)^2 d\mu - \int \frac{|\gr R|^2}{R} d\mu \leq 0.
\end{multline}
\end{proof}

\subsection{Uniform bounds for $R$.}

Now, we want to combine Hamilton's Li-Yau harnack inequality with the
entropy estimates to obtain uniform bounds for $R$.

Let $\Rmin(t) = \min_{x \in M}R(x, t)$ and
$\Rmax(t) = \max_{x \in M}R(x, t)$.

\noindent
{\bf Claim:} For any $t \in [\tau, \tau + (2 R_{max}(\tau)) ] = [\tau, T]$,
$$
\Rmax(t) \leq 2 \Rmax(\tau).
$$

To see this, recall that the evolution equation for the curvature is given by
$$
\ddt R = \De R + R^2 -rR.
$$
Hence, at a maximum in space (where $\Rmax > 0 $),
$$
\ddt R \leq \De R + R^2.
$$

By the maximum principle, we can compare $R$ with the solution of
the associated ODE, obtaining
$$
\Rmax (t) \leq \frac{1}{\Rmax^{-1} + \tau - t} \leq 2 \Rmax(\tau).
$$

This allows us to conclude that
$$
g(x, \tau) \leq e g(x, t)
$$
for any $t$ in the interval
$[\tau, \tau + (2 R_{max}(\tau)) ]$.
This follows from integrating the \RF equation:
$$
g(x, t) = e^K g(x, \tau),
$$
where
$$
K = \int_\tau^t (r-R) ds \geq \int_\tau^t -R ds
\geq -2 \int_\tau^T \Rmax(\tau) ds = -1.
$$

Hence, if $d$ is the geodesic distance at time $T$,
$$
\De(\xi, \tau, X, T) \leq C \frac{d^2(X, \xi)}{T - t}.
$$

Applying Hamilton's Li-Yau Harnack inequality, and noting that
$\frac{e^{r\tau -1}}{e^{rT} - 1}$ is a topological constant,
$$
R(\xi, \tau) \leq C R(X, T) \hspace{.5cm} \text{for all} \hspace{.5cm} x \in B\rho(\xi),
$$
where\footnote{This is an application of Klingenberg's Theorem.}
$$
\rho = \frac{\pi}{2 \Rmax(T)^2}.
$$

On the other hand, using the entropy estimate, we obtain
\begin{multline}
C \geq \int R \log R d\mu \geq \int_{B_\rho} R \log(c \Rmax(\tau)) d\mu\\
\geq c \int_{B_\rho} R d\mu + \log(c \Rmax(\tau)) c \int_{B_\rho} R d\mu \geq C \log(c \Rmax(T)),
\end{multline}
{\em i.e.}, $\Rmax (T)$ is bounded, hence $\Rmax(\tau)$
 is bounded. Since this is true for all $\tau > 1$, then $R$ is bounded.

Recall that a bound on $R$ gives a lower bound on the injectivity radius, and
hence an upper bound on the diameter (since the area is constant for the \NRF).
Then, if $T - \tau \leq 1$,
$$
\De(\xi, \tau, X, T)\leq \frac{C}{T - \tau}.
$$
Combining with Harnack, we obtain, for $t \geq 1$,
$$
R(x, t) \leq C R (y, t+1)
$$
for any two points $x, y$.

This completes the proof of the following result.
\bethm
For a solution  of the \RF with $R > 0$ on a compact surface,
there exist constants $0 < c < C < \infty$ such that,
for all times,
$$
c \leq R(t) \leq C.
$$
\eethm

We will refer the reader to \cite{H1} for the proof of the uniform bounds
of the derivatives of $R$. The proof is to estimate inductively the $L^2$-norms
of the derivatives, and to observe that the Sobolev constants may be taken uniformly.

\subsection{Convergence to a constant curvature metric}

The strategy to complete the proof of the main result of this chapter, that is,
to prove that the \RF on surfaces converge to a metric with constant curvature,
is to modify a bit the flow equation by a innocuous term, and note that the resulting
flow converges to a desired metric.

Recall the definition of the trace-free Hessian of $f$
$$
M_{ij} = \gr_i \gr_j f - \frac{1}{2} \De f g_{ij},
$$
where $f$ is the potential for the curvature, which satisfies
$$
\De f = R - r.
$$

\begin{lemma}
The evolution equation for $M$ is given by
$$
\ddt |M_{ij}|^2 = \De |M_{ij}| -2 |\gr_k M_{ij})|^2 - 2 R| M_{ij})|^2.
$$
\end{lemma}

\begin{proof}
\begin{eqnarray*}
\ddt M_{ij} &=& \gr_i \gr_j \ddt f -  (\ddt \Ga^k_{ij}) \gr_k f - \frac{1}{2} \ddt [(R - r) g_{ij}]\\
& = &
\gr_i \gr_j \De f +
\frac{1}{2}
\left(
\gr_i R \gr_j f + \gr_j R \gr_i f - \ip{\gr R, \gr f}g_{ij}
\right) \\
& &  - \frac{1}{2}\De R g_{ij}+ r M_{ij}.
\end{eqnarray*}

Recalling the formula for the commutator of $\De$ and $\gr\gr$, we obtain
$$
\ddt M_{ij} = \De \gr_i \gr_j f - 2RM_{ij} - \frac{1}{2} \De R g_{ij} + r M_{ij}.
$$
\end{proof}

Once again, we apply the maximum principle to get
\begin{cor}
If $R \geq c > 0$, then
$$
|M_{ij}| \leq C e^{ct}.
$$
\end{cor}

Now, we consider the modified flow
$$
\ddt g_{ij} = M_{ij} = (r- R )g_{ij} - 2\gr_i \gr_j f,
$$
which differs from the \RF only by transport along a family of
diffeomorphisms generated by the gradient vector field of $f$.
Note that $|M_{ij}|$ is invariant under diffeomorphisms.

We will
refer to \cite{Ben1} or \cite{H1} for the proof that also the derivatives of $M_{ij}$
decay exponentially to zero.

So, we proved that
the modified Ricci flow converges exponentially to a metric $g_\infty$ such that
the corresponding $M_\infty$ vanishes. Therefore, $g_\infty$ is a gradient Ricci soliton.
If we prove that the only gradient Ricci solitons on a compact surface are the trivial ones
(metrics with constant curvature), the diffeomorphism invariance tells us that
the solution to the \NRF on a compact surface with $R>0$  converges to a metric of constant curvature.

\begin{prop}
If $g(t)$
 is a soliton solution to \NRF on $S^2$, then
 $g(t) = g(0)$ is a metric with constant curvature.
\end{prop}

\begin{proof}
The equation for  a Ricci soliton is given by
$$
(r - R)g_{ij} = \gr_i X + \gr_j X.
$$
Contracting with $Rg^{-1}$, we get
$$
2R(r - R) = 2 R \text{div} X,
$$
and hence
$$
-\int_{S^2} (R-r)^2 d\mu = \int_{S^2}R \text{div} X d\mu.
$$

Now, using the Kazdan-Warner identity,  we obtain
$$
\int_{S^2} (R-r)^2 d\mu = \int_{S^2} \ip{\gr R, X} d\mu = 0,
$$
and hence $R = r$.
\end{proof}

An important remark is that the Kazdan-Warner identity is the
only place where we assumed the Uniformization Theorem for the sphere.
In fact, Chen, Lu and Tian \cite{XXChen-Lu-Tian} provide a proof that
the only gradient shrinking solitons on a sphere are the trivial ones
without using uniformization. Therefore, the \RF can be used to prove
the Uniformization Theorem.

\section{The case where $r> 0$, but $R$ is of mixed sign.}

In \cite{Bensphere}, Chow completes the proof of the main result of this chapter,
by proving the following theorem.
\bethm
If $g$ is any initial metric in $S^2$, then under the \NRF, the scalar curvature
$R$ becomes positive in finite time.
\eethm

Once the curvature $R$ becomes positive, we may "restart" the flow, and the results
on the previous section will guarantee the convergence to a constant curvature metric.

Chow's idea is to prove a modified Hamilton-Li-Yau Harnack inequality for the case
where $R$ is of mixed sign, and combine the flow with certain quantities that help to
control the negative parts of the curvature. We refer the reader to \cite{Bensphere}
for the complete, well-detailed exposition of the result.

\clearpage\mbox{}\clearpage

\chapter{Short-Time Existence}

As we have seen in Chapter~\ref{Introduction}, the Ricci flow is not
parabolic, so standard PDE techniques are not sufficient to guarantee that
(\ref{e.RF}) has a solution even for short time.

Hamilton, in \cite{H1}, first proved the short-time existence by the use
of a very fancy analytic tool, the Nash-Moser implicit function theorem,
and his proof was very deep and intruncated.
Soon after, DeTurck \cite{DeTurck} provided a much simpler proof, by noting that
the  \RF equation can be related to a modified parabolic system, in the sense that
solutions of the modified equation will be related to the original \RF.
That is the subject of this chapter.

\section{The linearization of the Ricci tensor}

We start by recalling a few definitions from the theory of
partial differential equations. In this section, the reader should
be warned that the multi-index notation is being used.

Let $\cal V$ and $\cal W$ be vector bundles over $M$, and let $L$
be a {\bf differential operator of order $k$} on $M$, that can
be written as
$$
L(V) = \sum_{|\alpha|\leq k} L_\alpha \partial^\alpha V,
$$
for $V \in C^\infty(\cal V)$, where $L_\alpha \in \text{Hom}(\cal V, \cal W)$.

\bede
We define the {\bf principal symbol}
of the linear differential operator $L$
in the direction of the covector $\xi$ by
$$
\si(L)(\xi) = \sum_{|\alpha|= k} L_\alpha (\Pi_j \xi^{\alpha_j}).
$$
\eede

We leave an easy exercise for the reader: what is the principal symbol for
the Laplacian $\De$ in $\RR^n$?

\bigskip

We now regard $\Ric (g)$ as a non-linear partial differential operator on the metric $g$:
$$
\Ric_g = \Ric(g) : C^\infty(S_2^+T^*M) \rightarrow C^\infty(S_2T^*M).
$$
Its linearization is given by
$$
[\Ric_g'(h)]_{jk} =
\frac{1}{2}g^{pq}\left(
\gr_q \gr_j h_{kp} +
\gr_q \gr_k h_{jp} -
\gr_q \gr_p h_{jk} -
\gr_j \gr_k h_{pq}
\right).
$$
The principal symbol $\si(\Ric_g')$ in the direction of $\xi$
is given by replacing $\gr_i$
by $\xi_i$:
$$
\left[ \si(\Ric_g')(\xi)(h)\right]_{jk} =
\frac{1}{2}g^{pq}\left(
\xi_q \xi_j h_{kp} +
\xi_q \xi_k h_{jp} -
\xi_q \xi_p h_{jk} -
\xi_j \xi_k h_{pq}
\right).
$$

\bede
A linear partial differential operator $L$ is said to be
{\bf elliptic} if the principal symbol $\si(L)(\xi)$ is an isomorphism
for any $\xi \neq 0$.

A non-linear partial differential operator $N$ is {\bf elliptic}
if its linearization $N'$ is elliptic.
\eede

Unfortunately, due to the invariance of the Ricci tensor with
respect to diffeomorphism, {\em i.e.},
\beeq
\label{e.diff}
\Ric(\phi^*(g)) = \phi^*(\Ric(g)),
\eeeq
we will see that the principal symbol of $\Ric$ has non-trivial kernel.

Following the notation in \cite{DeTurck}, we  define,
for any symmetric tensor $T \in S^2(T^*M)$,
\begin{eqnarray*}
tr(T) &=& g^{kl}T_{kl}; \\
G(T)_{ij} & = & T_{ij} - \frac{1}{2} (tr(T))g_{ij}; \\
\de(T)_i & = & -g^{jk}\gr_kT_{ij}
\end{eqnarray*}
As we see, the {\bf divergence $\de$ of a $2$-tensor} defines a map from symmetric $2$-tensors to $1$-forms.
So, we can define an $L^2$-adjoint $\de^*: T^*M \rightarrow S^2 T^*M$ by
$$
v \mapsto \de^*(v) = \frac{1}{2}(\gr_j v_{i} - \gr_i v_{j}) = {\cal L}_{v^\sharp} g.
$$

The total symbol of $\de^*_g$ in the direction of $\xi$ is given by
$$
\left[\si[\de^*_g](\xi)X \right]_{ij} = \frac{1}{2} (\xi_i X_j + \xi_j X_i).
$$

In order to show that the principal symbol of $\Ric$
is not an isomorphism, consider the composition
$$
\Ric_g' \circ \de^*_g : C^\infty(T^*M) \rightarrow C^\infty(S_2 T^*M),
$$
which is, {\it a priori}, a third-order partial differential operator. So, its
principal symbol should have degree $3$.
However, due to (\ref{e.diff}), we have
$$
\left[ \Ric_g' \circ \de^*_g (X) \right]_{ij} = \frac{1}{2}\left[{\cal L}_{X^\sharp}(\Ric_g)\right]_{ij},
$$
where $\cal L_{X^\sharp}$ is the Lie derivative in the direction of the vector $X^\sharp$, dual to $X$.

Note that the right-hand side of the equation above only involves one derivative of $X$,
hence its principal symbol is at most of degree $1$ in $\xi$. Therefore, we conclude that
the principal symbol of $\Ric_g' \circ \de^*_g $ is the zero map.

A property of the symbol, whose simple proof is left to the reader,
is that the symbol of a composition of two operators is the composition
of the symbols of the operators. Then,
\begin{eqnarray*}
0 &=& \si \left[ \Ric_g' \circ \de^*_g \right](\xi) \\
  & =& \si \left[ \Ric_g'\right](\xi) \circ
  \si \left[  \de^*_g \right](\xi)
\end{eqnarray*}

Therefore, the image of $\de^*_g$ is in the kernel
of the linearized Ricci operator. So,
$$
\dim \Ker\si \left[ \Ric_g'\right](\xi) \geq n,
$$
showing that the Ricci operator is not elliptic.

In fact,
we refer to \cite{Ben1} for the proof that the dimension
of that kernel is exactly $n$, which tells us that the
failure of the Ricci operator to being elliptic relies solely
on its diffeomorphism invariance.

DeTurck, in \cite{DeTurck}, made a very clever use of this fact:
he noticed that, by modifying the \RF equation by adding a term
which is the Lie derivative of a metric with respect to a
vector field, that in its turn depends on the metric.

The modified equation becomes parabolic, and its solution
can be pulled back to a solution to the original \RF equation
by a carefully chosen family of diffeomorphisms. This is the subject of next section.

\section{DeTurck trick}

Rewriting the linearization of the
$\Ric$ operator,
\beeq
\label{e.linricci}
\Ric'(g)(h) = \frac{1}{2} \De_L h - \de^*(\de G(h)),
\eeeq
where the first term is the {\bf Lichnerovicz Laplacian} for tensors,
given by
\beeq
\label{e.Lich}
(\De_L)_{jk} = \De h_{jk} + 2 g^{qp}R^r_{qjk}h_{rp} - g^{qp}R_{jp}h_{qk} - g^{qp}R_{kp}h_{jq}.
\eeeq
and the second term in (\ref{e.linricci}) is such that
symmetric squares of $1$-forms belong to the kernel of the
linearized Ricci operator.

DeTurck, in \cite{DeTurck2}, showed that for any choice  of symmetric
$2$-tensor $T$, the expression $\de^*(T^{-1}\de G(T))$ has always the same
symbol as the second term in the right-hand side of
(\ref{e.linricci}).
With this in mind, it is natural to consider
the modified operator
$$
Q(g) = \Ric(g) - \de^*(T^{-1}\de G(T)),
$$
which is elliptic! The obvious choice for a symmetric
tensor $T$ is to take the initial Riemannian metric $g_0$ on $M$.

Therefore, the parabolic  system
\beeq
\label{e.Q}
\begin{cases}
\ddt g = -2 Q(g)
\end{cases}
\eeeq
has a solution defined for short time.

Now, we need to argue that a solution of (\ref{e.Q})
can be translated to a solution of the \RF equation (\ref{e.RF}).

\begin{lemma}
Let $v(y,t)$ be a vector field on $M$.
Then, for small $t$, there exists a unique
family of diffeomorphisms $\phi_t: M \rightarrow M$ such that
$$
\begin{cases}
\ddt \phi_t(x) = v(\phi_t(x), t),\\
\phi_0(x) = \text{id}.
\end{cases}
$$
\end{lemma}

The proof of this lemma is left as an exercise to the reader,
as it is analogous to the standard ODE case.

Another necessary ingredient is the following result.
\begin{lemma}
Let $g(t)$  be a family of Riemannian metrics, and
let $\phi_t$ be the $1$-parameter diffeomorphism family related to
the vector field $v(y, t)$. Then
\beeq
\ddt \phi_t^*(g) (x) = \phi_t^*\left( \ddt g(\phi_t(x))\right) + 2
\phi_t^*\left( \de^* v^{\sharp} (\phi_t(x))\right),
\eeeq
where $\sharp$ is the map corresponding vector fields and $1$-forms.
Both the maps $\sharp$ and $\de^*$ are defined with respect to $g$.
\end{lemma}

Now we are in position of proving the main result of this section.
Choose $v$ to be the vector field dual to the $1$-form
$$
v^\sharp = -T^{-1}(\de G(T)),
$$
and let $\phi_t$ be the $1$-parameter family of diffeomorphisms associated to
$v$.

If $g$ is a solution to the modified equation (\ref{e.Q}), we have
\begin{eqnarray*}
\ddt \phi_t^*(g) & = & \phi_t^*\left( \ddt g \right) + 2
\phi_t^*\left( \de^* v^\sharp \right) \\
&=& \phi_t^*(-2Q(g))  + 2 \phi_t^*\left( \de^* v^\sharp \right)\\
&=& -2 \Ric (\phi_t^*(g)),
\end{eqnarray*}
that is, $\phi_t^*(g)$ is a solution to (\ref{e.RF}).

This completes the proof of the short-time existence of a solution to
the unnormalized \RF.

\clearpage\mbox{}\clearpage

\chapter{Ricci Flow in $3$ dimensions}
\label{RicciFlow3d}

The goal of this chapter is to present the proof of Hamilton's theorem
for $3$-dimensional Riemannian manifolds with positive Ricci curvature,
following the paper \cite{H2}.

\bethm
{\bf (Hamilton)}
Let $M$ be a compact $3$-dimensional manifold which admits a metric
with strictly positive Ricci curvature. Then, $M$ admits a metric of
constant positive curvature.
\eethm

In order to complete this task, we will study
thoroughly the \RF equation (\ref{e.RF}) on three-dimensional
manifolds.

\section{Evolution of curvatures under the \RF}

We state here the evolution equations for all the geometric objects
under \RF. For the proof of those expressions, we refer the reader to
\cite{Ben1}.

\bigskip

\begin{lemma}
Let $g(t)$ be a solution to \RF. Then,
\begin{enumerate}
\item
The Levi-Civita connection $\Ga(g)$ evolves by
\beeq
\ddt \Ga^k_{ij} = -g^{kl}(\gr_iR_{jl} + \gr_jR_{il} - \gr_lR_{ij}).
\eeeq
\item
The Ricci tensor $\Ric (g)$ evolves by
\beeq
\ddt R_{jk} = \De R_{jk} + \gr_j\gr_k R  - g^{pq}\left(\gr_q\gr_j R_{kp} + \gr_q\gr_k R_{jp} \right).
\eeeq
\item
The scalar curvature $R$ evolves by
\beeq
\ddt R = 2 \De R  - 2 g^{jk}g^{pq}\gr_q \gr_j R_{kp} + 2 |\Ric|^2.
\eeeq
\end{enumerate}
\end{lemma}

From the second Bianchi identity
$$
\gr_m R_{ijkl} + \gr_k R_{ijlm} + \gr_l R_{ijmk} = 0,
$$
we obtain a nice form for the evolution of $R$:
\beeq
\label{e.scalar}
\ddt R = \De R + 2 |\Ric|^2.
\eeeq

Note that $|\Ric|^2 \geq 0$, hence the maximum principle
applied to (\ref{e.scalar}) implies the following positivity result:

\begin{lemma}
Let $g(t)$ be a solution to \RF with initial condition $g_0$.

If the scalar curvature $R(g_0)$ of the initial metric is bounded below by
some constant $C$, then $R(g(t))\geq C$ for as long as the solution exists.
\end{lemma}

Now we analyze the Ricci tensor. Its evolution is given by
$$
\ddt \Ric_{ij} = \De_L \Ric_{jk},
$$
where $\De_L$ is the Lichnerovicz laplacian defined in the previous chapter.

Written in this form, one may be led to think that the maximum principle
could be applied to the equation above, and that would imply that positivity
of the Ricci curvature is preserved along the flow.
However, $\De_L$ contains terms coming from the whole Riemann curvature tensor,
so more work is need to guarantee that positivity  is preserved.

We note that in $3$ dimensions, the Weyl tensor \footnote{The Weyl tensor is the conformal part
of the curvature tensor.} vanishes identically, which allows the
Ricci curvature to determine completely the Riemann tensor:
$$
R_{ijkl} = R_{il}g_{jk} + R_{jk}g_{il} - R_{ik}g_{jl} + R_{jl}g_{ik} -\frac{R}{2}( g_{il}g_{jk} - g_{ik}g_{jl}).
$$

Hence, we can rewrite the evolution of the Ricci curvature as
\begin{eqnarray}
\label{e.Riccievol}
\ddt \Ric_{jk} &=&  \De R_{jk} + 3R R_{jk} - 6 g^{pq}R_{jp}R_{qk} + (2 |\Ric|^2 - R^2)g_{jk}\\
 &=& \De R_{jk} + Q_{jk},
\end{eqnarray}
where the tensor $Q_{jk}$ is defined by the expression above.

\section{Maximum Principle for tensors}

In order to check that $\Ric \geq 0$ is preserved along the flow,
we need
a version of the maximum principle that can be applied to tensors.

Let $u^k$ be a vector field, and let
$M_{ij}$, $g_ij$ and $N_{ij}$ be symmetric tensors
(that may depend on $t$)
on a compact manifold $M$.
Assume that $N_{ij} = p(M_{ij}, g_{ij})$ is a polynomial in
$M_{ij}$  formed by contracting products of $M_{ij}$ with itself
using $g_{ij}$.

\bede
The tensor $N_{ij}$
satisfy the {\bf null-eigenvector condition}
if for any nullvector $v^i$ of $M_{ij}$,
$$
N_{ij}v^iv^j \geq 0.
$$
\eede

\bethm
\label{t.maxtensor}
{\bf (Maximum Principle for tensors)}
Suppose that on $0 \leq t \leq T$,
$$
\ddt M_{ij} = \De M_{ij} + u^k \gr_k M_{ij} + N_{ij},
$$
where $N_{ij} = p(M_{ij}, g_{ij})$ satisfies the
null-eigenvector condition.

Then, if $M_{ij} \geq 0$ at time zero, then it remains
greater than or equal to zero for all $t \in [0,T]$.

\eethm

\begin{proof}
The strategy of the proof is to show that $M_{ij}$ is non-negative
on a small interval $0\leq t \leq \de$, where $\de$ is very small (but uniform).
Splitting the whole intervals in sizes smaller than such $\de$ will complete
the proof of the result.

We define
$$
\tilde M_{ij} = M_{ij}+ \ve (\de + t) g_{ij}.
$$

\noindent
{\bf Claim:} $\tilde M_{ij} > 0 $ on $0\leq t \leq \de$, for every $\ve> 0$.

Letting $\ve \rightarrow 0$ proves the theorem.

To prove the claim, assume not. Then, there exists a
first time $\theta$, $0< \theta \leq \de$ such that
$\tilde M_{ij}$ acquires a null  unit eigenvector $v^i$ at
some point $p \in M$.

Then, if $\tilde N_{ij} = p(\tilde M_{ij}, g_{ij})$, then
$\tilde N_{ij}v^i v^j \geq 0$ at $(p, \theta)$.
Moreover,
$$
|\tilde N_{ij} - N_{ij}| \leq C |\tilde M_{ij} - M_{ij}|,
$$
where $C$ only depends on $\max(|\tilde M_{ij}| + |M_{ij}|)$,
because $\tilde N$ and $N$ are just polynomials.
Furthermore, if we choose $\ve, \de < 1$, then
the constant $C$ only depends on $\max(|M_{ij}|)$.

\bigskip

We can extend the null eigenvector $v^i$ to a vector field on a
neighborhood of the point $p \in M$, such that
$\gr_j v^i = 0$ at $p$, and such that the vector field is not
time-dependent.

Write $f = \tilde M_{ij}v^i v^j$. Then, by our construction, $f\geq 0$
for all $0\leq t \leq \theta$ and all of $M$. Hence, since
$(p, \theta )$
is a minimum,
\beeq
\ddt f(p ) < 0, \hspace{1cm} \gr_kf = 0 \hspace{.5cm}\text{and} \hspace{.5cm} \De f \geq 0
\eeeq
at $(p, \theta)$.

Recalling that $v$ is not time-dependent, and
$g$ is a solution of the \RF,
\begin{eqnarray}
\label{e.abovef}
\ddt f &=& (\ddt M_{ij})v^iv^j + \ve - 2\ve(\de + t)R_{ij}v^iv^j \\
& \geq &
(\ddt M_{ij})v^iv^j  + \frac{\ve}{2},
\end{eqnarray}
provided that $\de \leq (8 \max |R_{ij}|)^{-1}$.

Also, at the point $(p, \theta)$,
\beeq
\gr_k f = \gr_k M_{ij} v^i v^j  \hspace{.5cm} \text{and} \hspace{.5cm} \De f = \De M_{ij} v^i v^j.
\eeeq
Plugging these in (\ref{e.abovef}), we see that
\begin{eqnarray}
0 > \ddt f &\geq &(\ddt M_{ij})v^iv^j  + \frac{\ve}{2} \\
& = &
\De M_{ij} v^i v^j + u^k\gr_k M_{ij} v^i v^j + N_{ij}v^iv^j + \frac{\ve}{2}.
\end{eqnarray}
The first term in the right-hand side of the equation above is non-negative,
while the second term vanishes. Hence, we conclude that
$$
-C \ve \de \geq N_{ij} v^iv^j < -\frac{\ve}{2},
$$
which produces a contradiction if $\de$ is chosen to be sufficiently small.
This proves that $M_{ij} \geq 0$ on $0 \leq t \leq \de$. To complete the proof,
just split the whole interval in pieces of length smaller than $\de$, and apply the
result for each interval, from left to right.
\end{proof}

\bigskip

Now, the result about preserving non-negativity of the Ricci curvature is a straightforward corollary.

\begin{cor}
Under the \RF on a $3$-dimensional manifold, if the initial metric
$g_0$ has non-negative Ricci curvature, then
$$
\Ric(t) \geq 0
$$
for as long as the solution exists.
\end{cor}

\begin{proof}
We apply the maximum principle for tensors, with
$M_{ij} = R_{ij}$, $u^k = 0$ and
$N_{ij}  = -Q_{ij}$,
where $Q_{ij}$ is the tensor defined in the evolution equation for $R_{ij}$
in (\ref{e.Riccievol}).
\end{proof}

We remark that we can use this form of the maximum principle for tensors
because $Q_{ij}$ is indeed a polynomial on $R_{ij}$, and this phenomenon
does not happen in dimensions higher than $3$.

\bigskip

A straightforward computation yields the following result:
\begin{lemma}
If $R \neq 0$,
$$
\ddt \left(\frac{R_{ij}}{R} \right)=
\De\left(\frac{R_{ij}}{R} \right)
+ \frac{2}{R} g^{pq}\gr_pR \gr_q\left(\frac{R_{ij}}{R} \right)
- \frac{RQ_{ij}+ 2 SR_{ij}}{R^2},
$$
where $S$ (as in \cite{H2}) is given by
$$
S = g^{il}R_{ij}g^{jk}R_{kl}.
$$
\end{lemma}

Now, we are in position of proving a nice upper bound for $\Ric$:

\bethm
If $R \geq 0$ and $R_{ij} \geq \ve R g_{ij}$ for some $\ve> 0$ at time zero,
then $R_{ij}(t) \geq \ve R g_{ij}(t)$ for all $t$ such that a solution to
\RF exists.
\eethm

\begin{proof}
We shall make use of the maximum principle for tensors once more.
We have already seen that $R>0$ is preserved along the flow.
Apply Theorem \ref{t.maxtensor} for
\begin{eqnarray*}
M_{ij} &=& \frac{R_{ij}}{R} - \ve g^{ij}\\
u^k &=& \frac{2}{R} g^{kl}\gr_lR \\
N_{ij} &=& 2\ve R_{ij} - \left( \frac{RQ_{ij}+ 2 SR_{ij}}{R^2}\right) .
\end{eqnarray*}

The previous lemma ensures that this $M_{ij}$ evolves like in Theorem \ref{t.maxtensor},
but we still need to see what happens to $N_{ij}$ when $M_{ij}$ acquires a null-eigenvector.

At this point, it is convenient to recall we can always diagonalize the Ricci
curvature at a point, and thanks to the {\bf Uhlenbeck trick}\footnote{For more details, see \cite{Ben1}.},
we can find a moving
frame along the $3$-manifold that preserves the orthonormal frame. The
reader should be warned once more that this is only possible in dimension
$3$, when all compact manifolds are parallelizable.
Say that the Ricci tensor can be written as
$$
\Ric =
\begin{pmatrix}
\la & 0 & 0 \cr
0& \mu & 0 \cr
0 & 0 & \nu
\end{pmatrix},
$$
such that $\la(0) \geq \mu(0) \geq \nu(0)$ for the initial metric.

If $\Ric$ is written in diagonal form, so are $M_{ij}$ and $N{ij}$.
Without loss of generality, suppose that the
null eigenvector of $M_{ij}$ is the first eigenvector, corresponding
to the eigenvalue $\la$ of $\Ric$. Since $R = \la+ \mu+ \nu$,
we obtain the equation $\la = \ve(\la+\mu+\nu)$.
The corresponding entry in $R^2 N_{ij}$ is given by
$$
2\ve(\la+\mu+\nu)^2 - (\la+\mu+\nu)(2\la^2 - \mu^2 - \nu^2 - \la\mu - \la\nu + 2\mu\nu)
- 2 \la(\la^2+\mu^2+\nu^2).
$$
Using the relation $\la = \ve(\la+\mu+\nu)$ to get rid of $\ve$, the equation
above can be simplified to
$$
(\la+\mu+\nu)[\la(\mu + \nu) + (\mu - \nu)^2]- 2 \la(\la^2+\mu^2+\nu^2).
$$

Let's put together what we know. By taking traces, our hipothesis tells us that
$R  \geq 3 \ve R$, and since $R > 0$, $\ve\leq 1/3$.
In this case,
$(\mu + \nu) \geq 2\la$, showing that our choice of $N_{ij}$ satisfies the
null-eigenvector condition.

The result then follows from Theorem \ref{t.maxtensor}.
\end{proof}

\section{Pinching the Eigenvalues}

In this section, we will see that after a while, the
eigenvalues of the Ricci tensor approach each order.
Intuitively, the manifold starts to become more and more round.

Consider the expression
$$
S - \frac{1}{3}R^2 = (\la - \mu)^2 + (\mu - \nu)^2 + (\la - \nu)^2,
$$
which measures how the eigenvalues are far away from each other.

If indeed the $3$-manifold is becoming spherical, one would expect that
$S - \frac{1}{3}R^2$ shrinks along the flow.

\bethm
\label{t.pinching}
Let $(M,g_0)$ be a $3$-dimensional Riemannian manifold, such that
$g_0$ is a metric with $\Ric \geq 0$.
Let $g(t)$ be a solution to the (unnormalized) \RF equation (\ref{e.RF}) on
$0\leq t \leq T$,
with initial condition $g(0) = g_0$.

Then, there exist constants $\de > 0$ and $C$, that only depend on $g_0$,
such that
$$
S - \frac{1}{3}R^2 \leq R^{2 - \de}
$$
on $0\leq t \leq T$.
\eethm

\begin{proof}
Let $\ga = 2 - \de$, and recall the equations for the evolution of $\Ric$ and $R$:
\begin{eqnarray}
\ddt R_{ij} & = & \De R_{ij} - Q_{ij}\\
\ddt R  &=& \De R + 2S
\end{eqnarray}

We define
$
T  = g^{in}g^{jk}g^{lm}R_{ij}R_{kl}R_{mn}
$,
and writing $S$ and $T$ in terms of the eigenvalues of $\Ric$, we have
$$
S = \la^2 + \mu^2 + \nu^2  \hspace{.5cm} \text{and} \hspace{.5cm} T = \la^3 + \mu^3 + \nu^3.
$$
Finally, let
$$
C = \frac{1}{2} g^{ik} g^{jk} Q_{ij} R_{kl} = \frac{1}{2}\left[ R^3 - 5RS + 6T\right].
$$

A simple computation gives the following result:

\begin{lemma}
The expression $S$ evolves by
$$
\ddt S = \De S - 2 |\De_iR_{jk}|^2 + 4(T - C).
$$
\end{lemma}

Using this lemma, we can state the evolution for the two terms
$\frac{S}{R^\ga}$ and $R^{2-\ga}$:
\begin{lemma}
\label{l.chato1}
If $R > 0$,
\begin{eqnarray*}
\ddt \left( \frac{S}{R^\ga}\right) & = &
\De  \left( \frac{S}{R^\ga}\right) + \frac{2(\ga - 1)}{R}
g^{pq}\gr_pR \gr_q  \left( \frac{S}{R^\ga}\right) \\
&& - \frac{2}{R^{\ga + 2}}|R \gr_iR_{jk} - \gr_i R R_{jk}|^2 \\
&&
- \frac{(2-\ga)(\ga - 1)}{R^{\ga + 2}} S |\gr_i R|^2
+ \frac{4R(T-C) - 2\ga S^2}{R^{\ga + 1}}
\end{eqnarray*}
\end{lemma}

\begin{lemma}
\label{l.chato2}
If  $R> 0$, then for any $\ga$,
\begin{eqnarray*}
\ddt R^{2 - \ga} & = &
\De (R^{2 - \ga})  + \frac{2(\ga - 1)}{R} g^{pq} \gr_p R \gr_q(R^{2 - \ga})\\
&& -
\frac{(2-\ga)(\ga - 1)}{R^{\ga + 2}}R^2 |\gr_i R|^2 + 2(2 - \ga)R^{1-\ga}S.
\end{eqnarray*}
\end{lemma}

The proof of Lemmas \ref{l.chato1} and \ref{l.chato2} amounts to computing directly the
terms involved and it is left to the reader.

Now, we want to use the maximum principle for the function
$$
f = \frac{S}{R^\ga} - \frac{1}{3}R^{2-\ga}.
$$
It is a direct consequence of the two lemmas above to see that the
evolution equation satisfied by $f$ is given by
\begin{eqnarray*}
\ddt f &=&
\De f + \frac{2(\ga - 1)}{R} g^{pq} \gr_p R \gr_q f -
\frac{2}{R^{\ga + 2}}|R \gr_iR_{jk} - \gr_i R R_{jk}|^2 \\
&& - \frac{(2-\ga)(\ga - 1)}{R^{\ga + 2}}\left(S - \frac{1}{3}R^2\right) |\gr_i R|^2
\\
&&
\frac{2}{R^{\ga + 1}}\left[ (2-\ga S )\left(S - \frac{1}{3}R^2\right) - 2P\right],
\end{eqnarray*}
where $P = S^2 + R(C - T)$.

Clearly, $P$ is a symmetric polynomial of degree $4$ in the eigenvalues $\la$, $\mu$ and $\nu$.
But in fact, we can say more:

\begin{lemma}
In terms of the eigenvalues of $\Ric$, $P$
 is given by
$$
P = \la^2 (\la - \mu)(\la - \nu)
+  \mu^2 (\mu - \la)(\mu - \nu)
+ \nu^2 (\nu - \la)(\nu - \mu).
$$
\end{lemma}

Now, we can state the following result.
\begin{lemma}
If $R > 0$ and $R_{ij} \geq \ve R g_{ij}$, then
$$
P \geq \ve^2 S \left(S - \frac{1}{3}R^2\right).
$$
\end{lemma}

To see this, first we simplify a bit: since both sides are
homogeneous of degree $4$, it suffices to check the statement
on $S = (\la^2 + \mu^2 + \nu^2)=1$.

Assume $\la \geq \mu \geq \nu > 0$. Since
$(\la + \mu + \nu)^ 2 \geq 1$,
$\nu \geq \ve(\la + \mu + \nu) \geq \ve$, because of the second condition
in the hypothesis of the lemma.
Furthermore, notice that
$$
P\geq \la^2 (\la - \mu)^2 + \nu^2(\mu - \nu)^2,
$$
which implies
$$
P \geq \ve^2[(\la - \mu)^2 + (\mu - \nu)^2].
$$

On the other hand, since
$$
(\la - \nu)^2 \leq 2 [(\la - \mu)^2 + (\mu - \nu)^2],
$$
we see that
$$
\left(S - \frac{1}{3}R^2\right)  = [(\la - \mu)^2 +(\la - \mu)^2 +(\mu - \nu)^2]
\leq [(\la - \mu)^2 + (\mu - \nu)^2] \leq P,
$$
which completes the proof of the lemma.

Finally, we can state the proposition which completes, via the maximum principle,
the proof of the theorem.
\begin{prop}
If $\de \leq 2\ve^2$, then
$$
\ddt f \leq \De f + u^k \gr_k f,
$$
for $u^k = \frac{2(\ga - 1)}{R}g^{kl}\gr_l R$.
\end{prop}
This proposition follows directly from our previous estimates.

Now we can finish the proof of Theorem \ref{t.pinching}: choose
a constant $C$ such that $f \leq C$ at time zero.
Then, the maximum principle says that this bound is preserved for as
long as $f$ exists. In other words,
$$
\left(S - \frac{1}{3}R^2\right)  \leq C R^{\ga},
$$
as claimed.
\end{proof}

This theorem tells us that as long as $\Ric$ has a nice lower bound for time
zero, then the \RF will make the manifold rounder, by collapsing the eigenvalues together.
\bigskip

For restrictions on time, we will not proceed with the higher-order estimates
for the Ricci tensor $\Ric$. For a very good exposition, we refer to
\cite{Ben1}, or even the original paper by Hamilton \cite{H2}.

\chapter{Introduction to \kahler Geometry}
\label{KahlerGeom}

In order to study the \kahler version of the Ricci Flow, that evolves metrics to
the unique  \kahler-Einstein metric in each \kahler class (whose interest goes far beyond
Mathematics),
we need to review some basic definitions and facts about complex manifolds.

\section{\kahler manifolds}

In this section, we define the basic objects we will be dealing with,
while studying the \KRF.

\begin{defin}
Let $M$ be an $n$-dimensional manifold. We say that
$M$ is a {\bf complex manifold} if it admits a system of
{\bf holomorphic coordinate charts}, that is, charts such that the
transition functions are biholomorphisms.
\end{defin}

 At each point $p$ in a complex manifold $M$, we can define a map
 $$
 J: T_p M \rightarrow T_pM, \hspace{1cm} J = d(z^{-1}\circ \sqrt{-1}\circ z),
 $$
where $z$ is a holomorphic coordinate defined on a neighborhood of $p$.

It is simple to check that $J^2 = -Id$, and we call $J$ an {\bf almost-complex structure}.
In fact, the definition of an almost-complex structure is more general: it
is simply a map $J \in End(TM)$ such that $J^2 = -Id$.
We see that a complex structure induced an almost-complex structure, but the converse is
not true. We say that the almost-complex structure $J$ is {\bf integrable} if there
exists an underlying complex structure which generates it.

The {\bf Newlander-Nirenberg Theorem} states that an almost-complex structure is
integrable if the {\bf Nijenhuis tensor}
$$
N(X, Y) = [JX, JY] - J[JX, JY]- J[X, JY] - [X,Y]
$$
vanishes identically.

Let $M$ be a compact, complex manifold of complex dimension
$n$, and consider $g$,
a hermitian metric defined on $M$.
Note that $g$ is a complex-valued sesquilinear form acting on
$TM \times TM$, and can therefore be written as
$$
g = S - 2\sqrt{-1} \om_g,
$$
where $S$ and $-\om$ are real bilinear forms.

If $(z_1, \dots, z_n)$ are local coordinates around a point $p\in M$,
we can write the metric $g$ as $\sum g_{i\jbar}dz^i \otimes d\zbar^j$.
Then, it is easy to see that in these coordinates
$$
\om_g = \frac{\sqrt{-1}}{2} \sum_{i, j = 1}^n g_{i\jbar}dz^i \wedge d\zbar^j.
$$
The form $\om_g$ is a real $2-$form of type $(1,1)$, and is called
the {\it fundamental form} of the metric $g$.

\begin{defin}
We say that a hermitian metric on a complex manifold is
{\bf \kahler} if its associated fundamental form $\om_g$
is closed, {\em i.e.}, $d \om_g = 0$.
A complex manifold equipped with a \kahler metric is
called a {\bf \kahler manifold}.
\end{defin}

Another characterization of a \kahler manifold $M$ is a manifold
equipped with an almost-complex structure $J$ and a metric $g$ such that
$g$ is $J$-invariant ({\em i.e.}, $g(JX,JY) = g(X, Y)$) and
$J$ is parallel with respect to the Levi-Civita connection of $g$.

The reader can check that the conditions on $J$ and $g$
of a \kahler manifold implies that the Nijenhuis tensor vanishes, so
any \kahler manifold is necessarily complex.

On a \kahler manifold \footnote{In fact, this whole paragraph holds for an
almost-complex manifold.} $M$, the complexified tangent bundle $TM_\CC = TM \otimes \CC$
has a natural splitting. If $p$ is a point in $M$, the extension of map $J$ to
$T_pM_\CC$ (as a complex-linear map) has $\sqrt{-1}$ and $-\sqrt{-1}$ as
eigenvalues, and we define the associated eigenspaces as $T^{1,0}_pM$
and $T^{0,1}_pM$. Hence, we have a decomposition
$$
TM_\CC = T^{1, 0} M \otimes T^{0,1}M.
$$

\bede
A differential $(p+q)-$form $\om$ is  {\bf of type $(p,q)$} if it is a
section of
$$
\La^{p,q}M = (\La^p T^{1, 0}) \wedge (\La^q T^{0, 1}).
$$
\eede

Let $z^i = x^i + \sqrt{-1} y^i$, $i = 1, \cdots, n$
be complex coordinates. Then, we define
$$
dz^i = dx^i + \ii dy^i  \hspace{0.5cm} \text{and} \hspace{0.5cm} d\zbar^j = dx^j - \ii dy^j,
$$
and $\frac{\partial}{\partial z^i}$, $\frac{\partial}{\partial \zbar^j}$, the dual of
$dz^i$ and $d\zbar^j$. From this definition, it is easy to see that
\begin{eqnarray*}
T^{1,0}M &=& \text{spam}\left\{ \frac{\partial}{\partial z^i} \right\}^n_{i=1} \\
T^{0,1}M &=& \text{spam} \left\{ \frac{\partial}{\partial \zbar^i}  \right\}^n_{i=1}
\end{eqnarray*}

The exterior differentiation
$$d: \La^{p,q}M \rightarrow \La^{p+1,q}M \oplus \La^{p,q+1}M $$
also splits according to this decomposition:
\begin{eqnarray*}
\partial:&& \La^{p,q}M \rightarrow \La^{p+1,q}M \\
\bar{\partial}:&& \La^{p,q}M \rightarrow \La^{p,q+1}M.
\end{eqnarray*}

Finally, if $M$ is \kahler, we recall that  the \kahler (closed) form
$\om_g$ is given, in local coordinates, by
$$
\frac{\sqrt{-1}}{2} \sum_{i, j = 1}^n g_{i\jbar}dz^i \wedge d\zbar^j,
$$
where $g_{i\jbar} = g(\frac{\partial}{\partial z^i}, \frac{\partial}{\partial \zbar^j})$.
\footnote{Here, we are abusing notation and writing $g$ for both the Riemannian metric
and its complex extension to $TM_\CC$}.

A last remark we make is that, due to $J$-invariance, the coefficients of type
$g_{ij}$ and $g_{\jbar\lbar}$ vanish identically. This cancelation phenomenon also
happens for some of the coefficients of the Riemann curvature tensor. This is the subject
of the next section.

\section{Curvature and its contractions on a \kahler manifold}

Many of the results in this section will be presented without a proof,
mainly because they involve direct calculations using the definition.

The {\bf Complex Christoffel symbols} are defined in analogy with the Riemannian
version. We denote by $\gr_\CC$ (for simplicity, just $\gr$)  the Levi-Civita
connection of the hermitian metric $g$.
\begin{eqnarray*}
(\gr)_{\ddzi}\frac{\partial}{\partial z^j} & = &
\left(
\Ga^l_{ij} \frac{\partial}{\partial z^l}
+ \Ga^{\lbar}_{ij} \frac{\partial}{\partial \zbar^l}
\right)\\
(\gr)_{\ddzi}\frac{\partial}{\partial \zbar^j} & = &
\left(
\Ga^l_{i\jbar} \frac{\partial}{\partial z^l}
+ \Ga^{\lbar}_{i\jbar} \frac{\partial}{\partial \zbar^l}
\right)
\end{eqnarray*}

A more useful expression for the Christoffel symbols is given by the
following lemma.

\begin{lemma}
In holomorphic coordinates, the Christoffel symbols are
$$
\Ga^k_{ij} = \frac{1}{2}g^{k\lbar}
\left(
\ddzi g_{j\lbar} + \frac{\partial}{\partial \zbar^j}g_{i\lbar}
-\frac{\partial}{\partial \zbar^j} \frac{\partial}{\partial \zbar^l}g_{ij}
\right)
=
g^{k\lbar} \ddzi g_{j\lbar}.
$$
and $\Ga^k_{ij} = \Ga^k_{ji}$

All the coefficients are zero, except for the ones of the form $\Ga^k_{ij}$
or $\Ga^{\kbar}_{\ibar \jbar}$.
\end{lemma}

Let $R(g) = R_{i\jbar k \lbar}$ be the coordinates of the $(4,0)$-Riemann curvature tensor of the metric $g$
written in holomorphic coordinates. It is useful to know the following expression, writing $R(g)$
in terms of the metric $g$.

\begin{lemma}
The components of the \kahler Riemann tensor are given by
$$
R_{i\jbar k \lbar} =
-\frac{\partial^2}{\partial z^i \partial \zbar^j} g_{k\lbar} + g^{u\vbar}\ddzi g_{k\vbar}
\frac{\partial}{\partial d\zbar^j} g_{u\lbar}.
$$
\end{lemma}

Note that the only non-vanishing terms have two barred indices exactly. The vanishing of the
others has to do with the fact that $R(X,Y)$ is $J$-invariant on a \kahler manifold.

We define the {\bf Ricci curvature tensor} of the metric $g$ as being the trace of the
Riemann curvature tensor. Its components in local coordinates can be written
as
\begin{equation}
\label{e.Riccidef}
\Ric_{k\lbar} =  \sum_{i, j = 1}^n g^{i\jbar} R_{i\jbar k \lbar} =
-\frac{\partial^2}{\partial z_k \partial \zbar_l} \log \det (g_{i\jbar}).
\end{equation}

The {\bf Ricci form} associated to $g$ can then be defined by setting
\beeq
\label{e.ricci}
\Ric = \sum_{i, j = 1}^n Ric_{i\jbar} dz^i \wedge d\zbar^j.
\eeeq
in local coordinates.

Finally, we define the {\bf Laplacian} acting on functions  to
be given by
$$
\De  = g^{i{\jbar}}\gr_i \gr_{\jbar} = g^{i\jbar} \frac{\partial^2}{\partial z^i \partial z^j}.
$$

Further expressions concerning commutators of $\gr_i$, $\gr_{\jbar}$ and $\De$ will be
introduced as needed in the following chapter.

\section{Ricci-flat metrics: the Calabi-Yau theorem}
\label{Monge-Ampere}

In order to motivate the search for special metrics in each \kahler class,
we shall start from its original motivation.

\bigskip

Recall that the coordinates of the Ricci tensor are given by (\ref{e.ricci}).
The{\bf Ricci form} associated to $g$ can then be defined by setting
$$
\Ric = \sum_{i, j = 1}^n Ric_{i\jbar} dz^i \wedge d\zbar^j.
$$
in local coordinates. In fact, a computation shows that
the Ricci form is given by
$$
\Ric = -\frac{\ii}{2\pi}\frac{\partial^2}{\partial z^i \partial z^j}(\log \det(g)).
$$

Now, given a metric $g$, we can define a matrix-valued $2$-form $\Om$ by
writing its expression in local coordinates, as follows
\begin{equation}
\label{e.curvform}
\Om_i^j = \sum_{i, p = 1}^n g^{j\pbar}  R_{i\pbar k \lbar} dz^k \wedge d\zbar^l.
\end{equation}
This expression for $\Om$ gives a well-defined matrix of $(1,1)$-forms, to
be called the {\bf  curvature form } of the metric $g$.

Following Chern-Weil Theory, we want to look at  the following expression
$$
\det\left(\text{Id} + \frac{t\sqrt{-1}}{2\pi} \Om \right) = 1 + t\phi_1(g) + t^2\phi_2(g) + 
\dots,
$$
where each $\phi_i(g)$ denotes the $i$-th homogeneous component of the left-hand side,
considered as a polynomial in the variable $t$.

Each of the forms $\phi_i(g)$ is a $(i,i)$-form, and is called the
{\bf $i$-th Chern form} of the metric $g$. It is a fact (see for example
\cite{Wells} for further explanations) that the cohomology class
represented by each  $\phi_i(g)$ is independent on the metric $g$, and hence it is
a topological invariant of the manifold $M$.
These cohomology classes are called the {\bf Chern classes} of $M$
and they are going to be denoted by $c_i(M)$.

{\bf Remark:} We can define more generally the curvature $\Om(E)$ of a
hermitian metric $h$ on
a general complex vector bundle $E$ on a complex manifold $M$.

Let $\nabla = \nabla(h)$ be a connection on a vector bundle $E \rightarrow M$.
Then the {\bf curvature} $\Om_E(\nabla)$ is defined to be the element
$\Om \in \Om^2(M, \text{End}(E, E))$ such that the $\CC$-linear mapping
$$
\Om : \Gamma(M, E) \rightarrow \Om^2(M, E)
$$
has the following representation with respect to a frame $f$:
$$
\Om(f) = \Om(\nabla, f) = d\theta(f) + \theta(f) \wedge \theta(f).
$$
Here, $\Gamma(M, E)$ is the set of sections of the vector bundle $E$,
$\Om^2(M, E)$ is the set of $E-$valued $2$-forms, and
$\theta(f)$ is the connection matrix associated with $\nabla$ and $f$
(with respect to $f$, we can write $\nabla = d + \theta(f)$).

Simiarly one defines
the Chern class $c_i(M, E)$ of a vector bundle and these are also independent on the
choice of the connection.
In fact, we use the expression ``Chern classes $c_i(M)$ of the manifold $M$''
meaning the Chern classes $c_i(M, TM)$
of the tangent bundle of M.

\bigskip

We will restrict our attention to the first Chern class $c_1(M)$ of the manifold $M$.
Note that the form $\phi_1(g)$
represents  the class $c_1(M)$ (by definition),
and that $\phi_1(M)$ is simply the trace of the
curvature form:
\begin{equation}
\label{e.phi1}
\phi_1(g) =  \frac{\sqrt{-1}}{2\pi} \sum_{i=1}^n \Om_i^i
= \frac{\sqrt{-1}}{2\pi} \sum_{i, p = 1}^n g^{i\pbar}  R_{i\pbar k \lbar} dz^k \wedge d\zbar^l.
\end{equation}

On the other hand, notice that the right-hand side of (\ref{e.phi1}) is equal to
$\frac{\sqrt{-1}}{2\pi} \Ric_{k\lbar}$, in view of (\ref{e.Riccidef}).
Therefore, we conclude that the Ricci form of a \kahler metric represents
the first Chern class of the manifold $M$.
A natural question that arises is: given a \kahler 
class
$[\om] \in H^2(M, \RR) \cap H^{1,1}(M, \CC)$ in a compact, complex manifold
$M$, and any $(1, 1)$-form $\Om$ representing $c_1(M)$, is that possible to
find a metric $g$ on $M$ such that $\Ric(g) = \Om$?

This question was addressed to by Calabi in 1960, and it was answered by Yau \cite{Yau} almost
$20$ years later.

\begin{thm}
{\bf (Yau, 1978)}
If the  manifold $M$ is compact and \kahler, then there exists a unique
\kahler metric $g$ on $M$ satisfying $\Ric(g) = \Om$.
\end{thm}

This theorem has a large number of applications in different
areas of Mathematics and Physics. Its proof is based on translating
the geometric statement into a non-linear
partial differential equation, as follows.

First fix a \kahler form $\om \in [\om]$ representing the previously given
\kahler class in $H^2(M, \RR) \cap H^{1,1}(M, \CC)$.
In local coordinates, we can write
$\om$ as $\om = g_{i\jbar}dz^i \wedge d\zbar^j$.

The $(1, 1)$-form $\Om$ is a representative for  $c_1(M)$, and
we have seen that $\Ric(\om)$
represents the same cohomology class as $\Om$.
Therefore, since $\Ric(\om)$  is also a  $(1, 1)$-form,
we have that,  due to the famous $\partial \bar \partial$-Lemma,
there exists a function $f$ on $M$ such that
$$
\Ric(\om) - \Om =  \frac{\sqrt{-1}}{2\pi} \partial \bar\partial f,
$$
where $f$ is uniquely determined after imposing the normalization
\begin{equation}
\label{e.integrability}
\int_M \left( e^f - 1 \right) \om^n = 0.
\end{equation}
Notice that $f$ is fixed once we have fixed $\om$ and $\Om$.

Again by the $\partial \bar \partial$-Lemma,
any other $(1, 1)-$form in the same cohomology class $[\om]$
will be written as
$\om + \frac{\sqrt{-1}}{2\pi} \partial \bar\partial \phi$, for some
function $\phi \in C^\infty(M, \RR)$.

Therefore, our goal is to find a representative
$\om + \frac{\sqrt{-1}}{2\pi} \partial \bar\partial \phi$ of the
class $[\om]$ that satisfies
\begin{equation}
\label{e.ric}
\Ric\left( \om + \frac{\sqrt{-1}}{2\pi} \partial \bar\partial \phi \right)
=
\Om =
\Ric(\om) - \frac{\sqrt{-1}}{2\pi} \partial \bar\partial f.
\end{equation}

Rewriting (\ref{e.ric}) in local coordinates, we have
$$
- \partial \bar\partial \log \det \left( g_{i\jbar} +
\frac{\partial^2 \phi}{\partial z_i \partial \zbar_j}
                                                  \right)
=
- \partial \bar\partial \log \det \left( g_{i\jbar}\right) - \partial \bar\partial f,
$$
or
\begin{equation}
\label{e.welldef}
 \partial \bar\partial \log
\frac{\det \left( g_{i\jbar} +
\frac{\partial^2 \phi}{\partial z_i \partial \zbar_j}
                                                  \right)}{\det \left( g_{i\jbar}\right)}
= \partial \bar\partial f.
\end{equation}

Notice that, despite of the fact that this is an expression given in
local coordinates, the term at the right-hand side of (\ref{e.welldef})
is defined globally.
Therefore, we obtain an equation
well-defined on all of $M$. In turn, this equation gives rise to the
following (global) equation
\begin{equation}
\label{e.MA100}
\left(\om + \frac{\sqrt{-1}}{2\pi}\partial \bar\partial \phi \right)^n = e^f \om^n.
\end{equation}

We shall also require positivity of the resulting \kahler form:
$$\left(\om + \partial \bar\partial \phi \right) > 0  \hspace{.2cm}
\text{ on} \hspace{.2cm}     M. $$
This equation is a non-linear partial differential equation of
Monge-Amp\`ere type, that is going to be referred to from now on as the Complex
Monge-Amp\`ere Equation.

We remark that, if $\phi$ is a solution
to (\ref{e.MA100}),  $\om + \partial \bar\partial \phi$
is the \kahler form of our target metric $g$, {\em ie},
$\Ric(g) = \Om$.
Therefore,  in order
to find metrics that are solutions to Calabi's problem, it suffices to determine a solution
$\phi$ to (\ref{e.MA100}).

The celebrated Yau's Theorem in \cite{Yau} determines a unique solution to (\ref{e.MA100})
when $f$ satisfies the integrability condition
(\ref{e.integrability}). The proof of this result is based on the continuity method,
and we sketch here a brief outline of the proof.

The uniqueness part of Calabi Conjecture was proved by Calabi in the $50$'s.
Let $\om', \om'' \in [\om]$ be representatives of the \kahler class $[\om]$
such that $\Ric(\om') = \Ric(\om'') = \Om$. Without loss of generality, we may assume
that $\om'' = \om$, and hence $\om' = \om + \partial \bar \partial u$.

Notice that
\begin{eqnarray}
\label{e.blah}
0 &=& \frac{1}{\Vol_\om(M)} \int_M u((\om')^n - \om^n) \\
&=&\frac{1}{\Vol_\om(M)} \int_M -u \partial \bar \partial u \wedge
\left[  (\om')^{n-1} + \right. \\
&& \left.(\om')^{n-2}\wedge \om + \cdots + \om^{n-1}   \right].
\end{eqnarray}

\noindent However $\om'$ is a \kahler form, so that $\om'>0$. We then conclude that
the right-hand side of (\ref{e.blah}) is bounded from below by
$$\frac{1}{\Vol_\om(M)} \int_M -u \partial \bar \partial u \wedge w^{n-1}.$$
Therefore,
\begin{eqnarray}
0 &\geq &\frac{1}{\Vol_\om(M)} \int_M -u \partial \bar \partial u \wedge w^{n-1}
\\
&=& \frac{1}{n\Vol_\om(M)} \int_M |\partial u|^2 w^{n}
\\
&=& \frac{1}{2n\Vol_\om(M)} \int_M |\nabla u|^2 w^{n},
\end{eqnarray}
implying that $|\nabla u| = 0$, hence $u$ is constant, proving the uniqueness
of solution to (\ref{e.MA100}).

\bigskip

Let us now consider the existence of solution to  (\ref{e.MA100}).
Define, for all $s \in [0,1]$, $f_s = sf + cs$, where
the constant $c_s$ is defined by the requirement that $f_s$ satisfies the
integrability condition $\int_M[e^{f_s} - 1 ]\om^n = 0 $.

Consider the family of equations
\begin{equation}
\label{e.fs}
\left( \om + \partial \bar \partial u_s\right)^n = e^{f_s} \om^n.
\end{equation}
We already prove that the solution $u_s$ to (\ref{e.fs}) is unique, if
it exists.

Let $A = \{s \in [0,1];$ (\ref{e.fs}) is solvable for all $t \leq s \}$.
Since $A \neq \emptyset$, we just need to show that $A$ is open and closed.

{\bf Openness:} Let $s \in A$, and let $t$ be close to $s$. We want to show that
$t \in A$. In order to do so, let $\om_s = \om + \partial \bar \partial u_s$, for $u_s$ a
solution to  (\ref{e.fs}).
We define the operator $\Psi = \Psi_s$ by
$$
\Psi: X \rightarrow Y; \hspace{2cm}
\Psi(g) = \log\left(\frac{(\om_s + \partial \bar \partial g)^n}{w_s^n}\right),
$$
where $X$ and $Y$ are subsets (not subspaces) of
$C^{2, 1/2}(M.\RR)$ and $C^{0, 1/2}(M.\RR)$
satisfying some extra non-linear conditions.

The linearization of $\Psi$ about $g = 0$ is simply the
metric laplacian with respect to the metric $\om_s$.
By the Implicit Function Theorem, the invertibility of the
laplacian (a result that can be found in \cite{GT}, for example)
establishes the claim.

{\bf Closedness:}
The proof that $A$ is closed is a deep result, involving complicated
{\em a priori} estimates. A reference for this proof is Yau's paper itself \cite{Yau},
or for a more detailed proof, the books \cite{Tianbook} and \cite{Asterisque}.

\medskip

 Yau's Theorem provided a satisfatory answer to the problem of
finding Ricci-flat metrics when  the underlying manifold $M$ is compact.

\section{\KE metrics}

The simplest examples of \kahler Ricci solitons are the (static) \KE metrics.

\bede
We say that a \kahler metric is {\bf \KE} if its Ricci form is a constant multiple of
the \kahler metric.
\eede

Since $\Ric$ represents the first Chern class, a
topological invariant, the existence of a \KE metric
implies that the Chern class of the manifold has a fixed sign.

If $c_1(M) < 0 $, Aubin and Yau \cite{Aubin-Yau} proved the Calabi conjecture
for negative first Chern class.

\bethm
{\bf (Aubin, Yau)} If a compact, complex manifold $M$ has $c_1(M) <0$, then
there exists  a \KE metric with negative scalar curvature. This metric
is unique up to scaling.
\eethm

The situation is far more complicated if $c_1(M) > 0$. In fact, there are
obstructions to the existence of a \KE metric in this scenario, namely  the so-called
{\bf Futaki Invariant}.

Fix a metric $\om$ such that $[\om] = c_1(M)$. The $\partial \bar\partial$-Lemma
implies that there exists a smooth function $f$ on $M$ such that
$$
\Ric - \om = \frac{\ii}{2\pi} \partial \bar \partial f.
$$

Let ${\cal H}(M)$ denote the set of holomorphic vector fields on $M$. We define the
{\bf Futaki functional} ${\cal F}_{[\om]} : {\cal H}(M) \rightarrow  \CC$ by
$$
{\cal F}_{[\om]}(V) = \int_M \ip{V, \gr f} d\mu.
$$

Futaki \cite{Futaki} proved that ${\cal F}_{[\om]}$ only depends on the
cohomology class $[\om]$, and its vanishing is a necessary condition for the
existence of a \KE metric.
Tian \cite{Tian-Futaki} showed, however, that the converse is not true: he
proved that some examples where ${\cal H}(M) =0$, but there are no
\KE metrics.

\clearpage\mbox{}\clearpage

\chapter{\kahler Ricci Flow}

This chapter is devoted to the study of the \RF on \kahler manifolds,
as first introduced by Cao \cite{Cao}.

Last chapter, we discussed the problem of finding \KE metrics on a compact manifold. Recall that
a necessary condition for a \kahler manifold to admit such metrics is that the first Chern Class
$c_1(M)$ has a sign.

The cases $c_1(M) = 0$ and $c_1(M) \leq 0$ were settled by  Yau \cite{Yau} and Aubin \cite{Aubin-Yau},
while we have seen that there are obstructions to the existence of a \KE metric if
$c_1(M) \geq 0$ (the non-vanishing of the Futaki invariant).

Here, the idea of Cao \cite{Cao} will be discussed in detail. By flowing any \kahler metric on a compact
\kahler manifold with either
$c_1(M) = 0$ or $c_1(M) \leq 0$
by the \RF, we obtain the (unique) \KE metric in the same \kahler class as the starting metric.

It should be noted here that Cao's proof relies in some results that are generalizations of Yau's
higher order estimates derived in \cite{Yau}.

\section{Settings}

Let
$$\Om = \frac{\ii}{2 \pi} T_{i\jbar} dz^i \wedge d\zbar^j$$
be a fixed representative of the
first Chern class $c_1(M)$, and denote, as before, the Ricci form of a \kahler metric $g$ by
$$
\Ric = \frac{\ii}{2 \pi} \Ric_{i\jbar} dz^i \wedge d\zbar^j = \frac{\ii}{2 \pi}  -\partial \bar \partial (\log \det(g)).
$$

The {\bf \kahler \RF equation } is
\beeq
\label{e.KRF}
\begin{cases}
\ddt {\tilde g}_{i\jbar}(t) = - {\tilde \Ric}_{i\jbar}(t) + T_{i\jbar} \\
{\tilde g}_{i\jbar}(0) = {g}_{i\jbar}
\end{cases}
\eeeq

If we can prove that the solution to (\ref{e.KRF}) exists for all times,
and converges to a limiting metric ${\tilde g}_{\infty}$, as well as show that
the derivatives $\ddt {\tilde g}_{i\jbar}(t)$ converge uniformly to a constant as $t$
approaches infinity, then ${\tilde g}_{\infty}$ is the \KE metric we want.

\subsection{Reduction to a scalar equation}

The beauty of \kahler geometry is that sometimes, all our study relies on
the \kahler potential. This is one of those happy situations.

Since both  $\Om$ and $\Ric$ lie in the same cohomology class ($c_1(M)$), the
$\partial \bar \partial$-Lemma tells us that there exists a smooth function $f$
on $M$ such that
$$
\Ric(\tilde g) - \Om = -\frac{\ii}{2\pi} \partial \bar \partial f.
$$

We are looking  for a metric $\tilde g = g +  -\frac{\ii}{2\pi} \partial \bar \partial u$,
for $u(t)$ defined on $M \times[0,t)$ with $u(0) = 0$, such that
\begin{eqnarray*}
\partial \bar \partial \left( \ddt u \right)
& = &
- {\tilde{\Ric}}_{i\jbar} + {\Ric}_{i\jbar} + \frac{\ii}{2\pi} \partial \bar \partial f \\
& = &
\partial \bar \partial \log \left(\frac{\det(g_{i\jbar} +  -\frac{\ii}{2\pi} \partial \bar \partial u)}{\det(g_{i\jbar})} \right)
+ \frac{\ii}{2\pi} \partial \bar \partial f
\end{eqnarray*}

Simplifying even more, the evolution equation for the \kahler potential $u$ is
\beeq
\label{e.KRpotential}
\ddt u =  \log \left(\frac{\det(g_{i\jbar} +  -\frac{\ii}{2\pi} \partial \bar \partial u)}{\det(g_{i\jbar})} \right) + f + c(t),
\eeeq
where $c(t)$ is a smooth function (on $t$) that satisfies the integrability condition
$$
\int_M e^{\ddt u - f} dV = e^{c(t)} \Vol(M).
$$

Equation (\ref{e.KRpotential}) is parabolic! So, short-time existence is guaranteed from
standard PDE techniques, as opposed to the (real) \RF equation.

In order to show long-time existence, we need to develop some {\em a priori} estimates
of the solution up to third order. With this in hand, a slight modification of Harnack inequality
will show that in fact $u(t) \rightarrow u_{\infty}$
uniformly, and that $\ddt u$ also converges uniformly to a constant.

Finally, we will discuss briefly the negative case of the Calabi Conjecture.

\section{Long-time existence}

Throughout this section, $u$ will denote the solution to the initial value problem
$$
\begin{cases}
\ddt u =  \log \det(g_{i\jbar}  -\frac{\ii}{2\pi} \partial \bar \partial u) - \log(\det(g_{i\jbar}))  + f \\
u(x, 0) = 0.
\end{cases}
$$
on the maximal time interval $[0, T)$, $T > 0$.

Differentiating this equation with respect to time, we obtain
$$
\ddt(\ddt u) = {\tilde g}^{i\jbar} \frac{\partial^2}{ \partial z^i \partial \zbar^j}
(\ddt u) = {\tilde \De}(\ddt u),
$$
and applying the maximum principle, we find out that
$$
\max_M |\ddt u| \leq \max_M |f|.
$$

\begin{lemma}
\label{l.lemma1}
Let $u_{\min} = \inf_{M \times [0,T)} u$. Then,
there exist constants $C_1, C_2 > 0$ such that
$$
0 < n + \De u \leq C_1 e^{C_2(u(t) - u_{\min} )}
$$
for all $t \in [0, T)$.
\end{lemma}

\begin{proof}
The first inequality comes from noticing that for all times, ${\tilde g}(t)$
is positive definite, and hence
$$
{tr}_g ({\tilde g}) >0.
$$

For the second inequality, we refer the reader to \cite{Yau}, equation (2.24),
which can be modified simply by considering the operator $\tilde \De - \ddt$
instead of the Laplacian.
\end{proof}

\subsection{Zeroth order estimates}
Now, we proceed to the zeroth-order estimates for $u$. Let $v = u - {\bar u}$,
where
$$
{\bar u} = \frac{\int_M u dV}{\Vol(M)}.
$$

Again, the following lemma can be derive simply from \cite{Yau}.
For a more detailed exposition, we refer the reader to \cite{Santoro2},
which contains an analog of this result in the context of open manifolds.
\begin{lemma}
\label{l.Yau}
{\bf (Yau)}
There exist positive constants $C_3, C_4$ such that
$$
\sup_{M \times [0, T)} v \leq C_3, \hspace{2cm} \sup_{M \times [0, T)} \int_M |v| dV \leq C_4.
$$
\end{lemma}

\begin{prop}
\label{p.supnorm}
There exists a constant $C$
such that
$$
\sup_{M \times [0, T)} |v| \leq C.
$$
\end{prop}

\begin{proof}
The reader should be warned that different constants may be denoted by the same
character $C$. Also, all the geometric objects related to ${\tilde g}$ will be
denoted with a $\sim$ on top. Namely, $\tilde \om $ will denote the \kahler form
associated to the metric $\tilde g$, and so on.

The strategy of the proof is to use Nash-Moser iteration process: we want to bound
$L^p$-norms of the function $v$ by lower $L^p$-norms, inductively. Together with
Lemma \ref{l.Yau}, this will imply the proposition.

\medskip

The volume forms of the metrics $g$ and $\tilde g$ are given, resp., by
$$
dV = \frac{\om^n}{n!}, \hspace{1cm} d{\tilde V} = \frac{{\tilde \om}^n}{n!}.
$$
From the evolution equation of $u$, we know that
$$
\ddt u - f = \log(\frac{{\tilde \om}^n}{{\om}^n}),
$$
and hence
$$
d{\tilde V} = e^{\ddt u - f} dV.
$$

Therefore, for $p> 1$,
$$
- \frac{1}{n!} \int_M \frac{(-v)^{p-1}}{p-1}(\om^n - {\tilde \om}^n) =
\int_M \frac{(-v)^{p-1}}{p-1}\left[ e^{(\ddt u -f)} - 1 \right]dV.
$$

On the other hand,
\begin{eqnarray*}
-  \int_M \frac{(-v)^{p-1}}{p-1}(\om^n - {\tilde \om}^n) & = &
\int_M (-v)^{p-2} (\frac{\ii}{2 \pi} \partial \bar \partial v) \wedge \sum_{j=1}^{n-1} \om^j\wedge {\tilde \om}^{n-j-1}
\\
& \leq &
\int_M (-v)^{p-2} (\frac{\ii}{2 \pi} \partial v \wedge \bar \partial v)\wedge \om^{n-1},
\end{eqnarray*}
where the last inequality follows from the fact that all the other terms involving $\om^j\wedge {\tilde \om}^{n-j-1}$
are non-negative.

Hence,
\beeq
\int_M (-v)^{p-2} |\gr v|^2 dV \leq C \int_M
\leq \int_M \frac{(-v)^{p-1}}{p-1}\left[ e^{(\ddt u -f)} - 1 \right]dV,
\eeeq
and so
\beeq
\int_M |\gr(-v)^{p/2}|^2 dV \leq C \frac{p^2}{p-1} \int_M  (-v)^p dV.
\eeeq

So, we can estimate the $H^1$-norm of $(-v)^{p/2}$ by
$$
|| (-v)^{p/2}||^2_{H^1} = || \gr (-v)^{p/2}||^2_{L^2} + || (-v)^{p/2}||^2_{L^2}
\leq C \frac{p^2}{p-1} \int_M  (-v)^p dV,
$$
and the last term can be bounded above by $C {p^2} \int_M  (-v)^p dV$ if $p>1$.

Now, since the function $v$ has zero average, Sobolev inequality implies that
$$
|| (-v)^{p/2}||^2_{L^{\frac{2n}{n-1}}} \leq|| (-v)^{p/2}||^2_{H^1}.
$$
Combining the last two inequalities, we obtain, for $p>1$,
\beeq
\label{e.iteration}
|| v||^p_{L^{p\frac{n}{n-1}}} \leq C p || v||^p_{L^p}.
\eeeq

The iteration happens here: let $\ga = \frac{n}{n-1}$, and replace
$\ga^j$ for $p$ in (\ref{e.iteration}), for $j = 0, 1, \cdots$.
By letting $j \rightarrow \infty$, and use the bound from Lemma \ref{l.Yau},
we obtain
$$
||v||_{L^\infty} \leq C,
$$
where the constant $C$ is independent of time. This completes the proof of the proposition.
\end{proof}

\subsection{Higher order estimates}

Combining Proposition \ref{p.supnorm} with Lemma \ref{l.lemma1}, we obtain a uniform bound for the Laplacian
of $u$ with respect to $g$: for some constant $C$,
\beeq
\label{e.laplacian}
n + \De v \leq C e^{C (u - \bar u)} =  C e^{C (v - \inf_{M \times [0,T)} v)} \leq C.
\eeeq
Applying Schauder estimate\footnote{A chapter about Schauder Theory
can be found in  \cite{GT}.}, we have the first-order estimate
$$
\sup |\gr v| \leq C \left(  \sup |\De v| + \sup |v| \right) \leq C,
$$
where the supremum in the expression above is taken over $M \times [0,T)$.

Also, note  that
$$
d{\tilde V} = e^{\ddt u - f} dV
$$
implies that the determinant of the complex Hessian of $u$ is uniformly bounded,
and (\ref{e.laplacian}) tells us that the trace of the Hessian of $u$ is also bounded
above. This gives us  the second-order bound on $u$, and we remark that this shows that
all metrics $\tilde g$ are uniformly equivalent to $g$.

Due to a limitation in time, we will refer the reader for the original paper of Cao
\cite{Cao} for the proof of the third order {\em a priori} estimates of $v$.
The method is a modification of the estimates in Yau's paper \cite{Yau}.

We are now in  position to prove the long-time existence of the \kahler Ricci potential $u$.

\begin{thm}
\label{t.longtime}
Let $u$ be a solution to (\ref{e.KRpotential}) on the maximum interval $[0, T)$, and
let $v = u - {\text Ave}_M u$.

Then, the $C^\infty$ norm of $v$ is uniformly bounded for all $t$, and therefore $T = \infty$.
\end{thm}

\begin{proof}
The use of Schauder theory for the heat operator $\De - \ddt$, together
with the estimates we derived in this section, will allow us to obtain
all the estimates via bootstrapping.

Differentiating (\ref{e.KRpotential}) with respect to $z^k$, we obtain
\beeq
\label{e.long-time}
\left( \tilde \De - \ddt \right)(\frac{\partial u}{\partial z^k})
=
g^{i\jbar} \frac{\partial }{\partial z^k} {\tilde g}_{i\jbar}
+
{\tilde g}^{i\jbar} \frac{\partial }{\partial z^k} {g}_{i\jbar}.
\eeeq

The coefficients of the operator $\tilde \De - \ddt$
are $C^{0,\alpha}$-bounded, as well as the right-hand side of
(\ref{e.long-time}). Therefore (see \cite{GT}), we see that
for all $k$, $\frac{\partial u}{\partial z^k}$ is in $C^{2, \alpha}$,
and similarly,  $\frac{\partial u}{\partial \zbar^k}$ is also in $C^{2, \alpha}$.

But this implies that, in fact, the coefficients of the operator $\tilde \De - \ddt$
and the right-hand side of (\ref{e.long-time}) are $C^{1,\alpha}$-bounded. Schauder
again implies that
$\frac{\partial u}{\partial z^k}$ and
$\frac{\partial u}{\partial \zbar^k}$ are in $C^{3, \alpha}$.

By iteration, we conclude that  the $C^\infty$-norm of $v$
is uniformly bounded, which shows that the \kahler Ricci
potential $u$ is defined for all times.
\end{proof}

\section{Uniform Convergence of the potential $u(t)$}

This section will be devoted to the proof of
the uniform convergence of the normalized potential $v$,
as well as to show that $\ddt u$ converges to a constant as
$t \rightarrow \infty$.

In \cite{Cao}, Cao used a slight generalization of the Li-Yau
Harnack inequality developed in \cite{Li-Yau}. We state Cao's version
here without a proof, since it can be derived simply from the result in \cite{Li-Yau}.

\begin{thm}
\label{t.Li-Yau-KRF}
{\bf (Cao)}
Let $M$ be an $n$-dimensional  compact manifold
and let $g_{i\jbar}(t)$, $t \in [0, \infty)$, be a family of \kahler metrics satisfying:
\begin{eqnarray*}
&(i)& C g_{i\jbar}(0) \leq g_{i\jbar}(t) \leq C^{-1} g_{i\jbar}(0); \\
& (ii)& |\ddt g_{i\jbar}|(t) \leq C g_{i\jbar}(0); \\
& (iii)& R_{i\jbar}(t) \geq -K g_{i\jbar}(0),
\end{eqnarray*}
where $C$ and $K$ are positive constants independent of $t$.

Let $\De_t$ be the Laplacian associated to the metric $g(t)$.

If $\vp$ is a positive solution for the equation
$$
(\De_t = \ddt ) \vp(x, t) = 0
$$
on $M \times [0, \infty)$, then
for any $\alpha > 1$,
\begin{eqnarray*}
\sup_{M}\vp(x, t_1) &\leq &  \inf_M \vp(x,t_2) \left(\frac{t_1}{t_2}\right)^{\frac{n}{2}}
\exp \left[ \frac{1}{4(t_2 - t_1)} C^2 d^2 \right. \\
& & \left. + \left( \frac{n \alpha K}{2(\alpha -1)} + C^2 (n + A) \right)
(t_2 - t_1)
\right],
\end{eqnarray*}
where $d$ is the diameter of $M$ measured using the initial metric,
$A = \sup |\gr^2 \log \vp|$ and $t_2 > t_1$.
 \end{thm}

Let $F = \ddt u$, where $u$ is the \kahler Ricci potential.
Then, $F$ is a solution to
\beeq
\label{e.parab}
\begin{cases}
\left( \De_t - \ddt  \right) F  = 0 \\
F(x, 0) = f(x)
\end{cases}
\eeeq
By the maximum principle, we have that, for times $t_2 > t_1$,
\begin{eqnarray*}
\sup_M F(x, t_2) &< \sup_M F(x, t_1) <& \sup_M f(x) \\
\inf_M F(x, t_2) &> \inf_M F(x, t_1) >& \inf_M f(x)
\end{eqnarray*}

In addition, we note that the family of metrics $\tilde g$ satisfies the
condition in  Theorem \ref{t.Li-Yau-KRF}.

Let's define auxiliary functions
\begin{eqnarray}
\vp_n (x, t) &=& \sup_M F(x, n-1) - F(x, n-1 + t)\\
\psi_n (x, t) &=& F(x, n-1 + t) - \inf_M F(x, n-1)\\
\om_n (x, t) &=& \sup_M F(x, t) - \inf_M F(x,  t),
\end{eqnarray}
which are all positive functions and satisfy (\ref{e.parab}). Applying
Theorem~\ref{t.Li-Yau-KRF} to $\vp$ and $\phi$ (for times $t_2 = 1, t_1 = 1/2$, we obtain
$$
\om(n-1) + \om(n - \frac{1}{2}) \leq (\om(n-1) - \om (n)),
$$
which implies that
\beeq
\label{e.osc}
\om(n)\leq \de^n (\sup_M f - \inf_M f),
\eeeq
where $\de = \frac{\ga - 1}{\ga} < 1$.

Since the oscillation function is decreasing on $t$, we
conclude from (\ref{e.osc}) that for $\de = e^{-a}$,
$$
\om(t) \leq C e^{-at}.
$$

Now, we defined the normalized derivative
$$
\phi = \ddt u - (\Vol(M))^{-1} \int_M \ddt u d \tilde V.
$$
The evolution equation for $\phi$ is given by
\begin{eqnarray*}
\ddt \phi & = &
\frac{\partial^2 u}{\partial t^2} -
(\Vol(M))^{-1} \int_M\left[ \frac{\partial^2 u}{\partial t^2}  + \ddt u {\tilde \De}(\ddt u)\right] d\tilde V \\
& = &
{\tilde \De}\left(\ddt u \right) -
(\Vol(M))^{-1} \int_M \ddt u {\tilde \De}(\ddt u)d\tilde V,
\end{eqnarray*}
where the first inequality follows from recalling the evolution equation for $d\tilde V$:
$$
\ddt d\tilde V = \ddt \log \det({\tilde g}_{i\jbar}) d\tilde V = \tilde \De (\ddt u) d\tilde V.
$$

Consider the quantity
$$
E = \frac{1}{2} \int_M \phi^2 d\tilde V.
$$
Some computations yield
$$
\ddt E = \int_M(\phi -1)|\tilde \gr \ddt u|^2 d\tilde V \leq \frac{1}{2}\int_M |\tilde \gr \phi|^2 d\tilde V,
$$
where last inequality follows from the fact  that
$$
\sup_M \phi < \om(t)< \frac{1}{2}
$$
for $t$ sufficiently large.

Poincar\'e inequality applied to $phi$ tells us that
$$
 \la_1(t)\int_M \phi^2 d\tilde V  \leq\int_M |\tilde \gr \phi|^2 d\tilde V,
$$
where $\la_1$ is the first eigenvalue of $\tilde \De(t)$. Note also that due to
the uniformity of the metrics $\tilde g(t)$, there exists a constant $c$ (independent of $t$)
such that
$\la_1(t) < c$.
This implies that
$$
\ddt E \leq -c E,
$$
and hence,
$$
\int_M \phi^2 dV \leq C e^{-ct},
$$
because all $d\tilde V$ are uniformly equivalent to $dV$.

\begin{prop}
The sequence of functions $v(x, t)$ (as defined in the previous section) converges
uniformly, as $t \rightarrow \infty$, to a smooth function $v_\infty$ on $M$. Furthermore,
$\ddt u$ also converges to a constant.
\end{prop}

\begin{proof}
We will show that the family $v(x,t)$ is Cauchy in the $L^1$-norm to some function $h(x)$.
Recall that  Theorem~\ref{t.longtime} implies that for some time sequence
$t_k \rightarrow \infty$, $v(x, t_k)$ converges to a smooth function $v_\infty(x)$.
So, $h$
has no choice but to be equal to $v_\infty$.

\bigskip

\noindent
{\bf Claim: } $v(x, t)$ is Cauchy in the $L^1-$ norm.

\medskip

To show this, let $0< s < \tau$, and consider
\begin{eqnarray*}
\int_M |v(x, s) - v(x, \tau)| dV & \leq &
\int_M \int_s^\tau |\ddt v|dt dV \\
& = &
\int_s^\tau \int_M |\ddt u - \frac{1}{\Vol M} \int_M \ddt u| dt dV \\
& \leq &
(\Vol(M))^{1/2} \int_s^\infty [\int_M \phi^2 dV]^{1/2} dt \\
& + &(\Vol (M))^{-1} \int_s^\infty \om(t) dt\\
& \leq &
C \int_s^\infty [e^{-C t}+ e^{-at} dt .
\end{eqnarray*}

At this point, we have seen that
$v(x,t)$ converges in  $L^1$-norm to the smooth $v_\infty$. It is not
hard to see that, in fact, the convergence happens in the $C^\infty$ norm.
Furthermore, it also follows that $\ddt u$ converges to a constant in the
$C^\infty$-norm.

\end{proof}

\section{Convergence to a \KE metric}

At this point, we proved all the ingredients to complete the proof of the main
theorem of this chapter.

\begin{thm}
Let $(M, g)$
 be  an $n$-dimensional  compact \kahler manifold, and let
 $\Om$ be a representative of the first Chern class $c_1(M)$ of
 $M$.

 Then, by deforming the initial metric $g$
via the flow
\beeq
\begin{cases}
\ddt g(t) = - \Ric(t) + \Om \\
g(0) = g
\end{cases}
\eeeq
we obtain another \kahler metric $\tilde g$, in the same \kahler
class as $g$, such that the Ricci curvature $\Ric(g)$ equals $\Om$.
\end{thm}

A straightforward consequence is the following.

\begin{cor}
If $c_1(M) = 0$, the \kahler \RF evolves any starting
\kahler metric to the Ricci-flat \kahler metric in its \kahler class.
\end{cor}

The proof of the main theorem is immediate after proving the uniform convergence
of the \kahler potential $u(t)$, and follows exactly like in Chapter~\ref{KahlerGeom},
Section~\ref{Monge-Ampere}.

Finally, we remark that this method can be applied to the problem of finding
\KE metrics on manifolds with $c_1(M) < 0$. The
evolution equation for this case is given by
$$
\ddt g(t) = - \Ric(g(t))  - g(t),
$$
where $g(0) = g$, and the initial metric $g$ represents $-c_1(M)$.

The equation for the \KE potential is
$$
\ddt u = \log \det (g + \partial \bar \partial u) - \log \det (g) + f - u.
$$
The analysis for this case is simplified, as the maximum principle
can be applied directly to the equation above, providing the zero-order estimate
for $u$, and a bootstrapping argument can be used, as before, for the higher-order
{\em a priori} estimates.

\bibliography{Bibsantoro}{}
\bibliographystyle{plain}

\end{document}